\documentclass[reqno]{amsart}
\pdfoutput=1
\pdfmapfile{-mpfonts.map}
\usepackage{amsrefs}
\usepackage{amssymb}
\usepackage{comment}
\usepackage{etex}
\usepackage[all]{xy}
\usepackage{mathtools}
\usepackage{euscript}

\usepackage{xcolor}
	\newcommand{\Sylvain}[1]%
	{\begin{quotation}{\footnotesize\textcolor{blue}{\textbf{Sylvain:} #1}}\end{quotation}}
	\newcommand{\THOMAS}[1]%
	{\begin{quotation}{\footnotesize\textcolor{brown}{\textbf{Thomas:} #1}}\end{quotation}}
	\newcommand{\Dani}[1]%
	{\begin{quotation}{\footnotesize\textcolor{red}{\textbf{Dani:} #1}}\end{quotation}}

\renewcommand{\epsilon}{\varepsilon}

\theoremstyle{theorem}
\newtheorem{theorem}{Theorem}
\newtheorem{corollary}{Corollary}[section]
\newtheorem{lemma}[corollary]{Lemma}
\newtheorem{proposition}[corollary]{Proposition}

\theoremstyle{remark}
\newtheorem{remark}[corollary]{Remark}

\theoremstyle{definition}
\newtheorem{definition}[corollary]{Definition}
\newtheorem{example}[corollary]{Example}

\usepackage{tikz}
\usetikzlibrary{arrows.meta}
\usetikzlibrary{patterns}
\usetikzlibrary{intersections}
\usetikzlibrary{matrix}
\usepackage{pgfplots}
\pgfplotsset{compat=1.14}
\tikzset{
  partial ellipse/.style args={#1:#2:#3}{
    insert path={+ (#1:#3) arc (#1:#2:#3)}
  }
}
\tikzset{
  partial ellipsecake/.style args={#1:#2:#3:#4}{
    insert path={+ (#1:#3) arc (#1:#2:#3 and #4) -- (0,0)  -- (#3,0)}
  }
}
\makeatletter
\tikzset{
  use path for main/.code={%
    \tikz@addmode{%
      \expandafter\pgfsyssoftpath@setcurrentpath\csname tikz@intersect@path@name@#1\endcsname
    }%
  },
  use path for actions/.code={%
    \expandafter\def\expandafter\tikz@preactions\expandafter{\tikz@preactions\expandafter\let\expandafter\tikz@actions@path\csname tikz@intersect@path@name@#1\endcsname}%
  },
  use path/.style={%
    use path for main=#1,
    use path for actions=#1,
  }
}
\makeatother
\usepackage{tikz-cd}

\DeclareMathOperator{\coker}{Coker}
\DeclareMathOperator*{\colim}{colim}
\DeclareMathOperator{\crit}{crit}
\DeclareMathOperator{\supp}{supp}

\newcommand{\tnabla}{\widetilde{\nabla}}
\newcommand{\T}{\mathcal{T}}
\newcommand{\del}{\partial}

\newcommand{\FF}{\mathcal{F}}
\newcommand{\HH}{\mathcal{H}}
\newcommand{\bC}{\mathbf{C}}
\newcommand{\bZ}{\mathbf{Z}}
\newcommand{\bR}{\mathbf{R}}
\newcommand{\E}{\mathcal{E}}
\newcommand{\NL}{\mathrm{NL}}
\newcommand{\R}{\mathbf{R}}

\newcommand{\B}{\mathcal{B}}
\newcommand{\C}{\mathbf{C}}
\newcommand{\Z}{\mathbf{Z}}

\newcommand{\N}{\mathbf{N}}
\newcommand{\TT}{\widetilde{\mathcal{T}}}
\newcommand{\GG}{\mathcal{G}}
\newcommand{\QQ}{\mathcal{Q}}

\newcommand{\OO}{\mathcal{O}}
\renewcommand{\SS}{\mathcal{S}}
\DeclareMathOperator{\MV}{MV}

\DeclareMathOperator{\sing}{Sing}
\newcommand{\TOP}{\mathrm{Top}}
\DeclareMathOperator{\id}{id}

\author{M. Abouzaid \and D. \'Alvarez-Gavela \and S. Courte \and T. Kragh}
\title{Normal invariant of nearby Lagrangians via twisted derivative} 
\date{\today} \thanks{}

\newcommand{\quadfunc}{2-homogeneous function}
\renewcommand{\quadfunc}{quadratic function}

\begin{document}

\maketitle

 \begin{abstract}
Let $L$ and $M$ be closed, connected, smooth manifolds and let $L \hookrightarrow T^*M$ be an exact Lagrangian embedding. The induced map $L \to M$ is known by earlier work to be a homotopy equivalence. We show that the associated normal invariant $M \to G/O$ factors through a map $B(\T,\QQ) \to G/O$ which is a twisted version of the Waldhausen derivative $\T \to G$ on the space $\T$ of tubes. Further, we show that this twisted derivative map itself factors though a map $B(G/O) \to G/O$ which is a twisted version of the $S$-duality map $BG \to G$. In particular we deduce that the normal invariant of the homotopy equivalence $L \to M$ is 2-torsion. \end{abstract}

\tableofcontents

\section{Introduction}\label{sec:intro}

\subsection{Main results}

Let $L$ and $M$ be closed connected smooth manifolds and let $L \hookrightarrow T^*M$ be an exact Lagrangian embedding.
We call $L$ a \emph{nearby Lagrangian} for short. The well-known \emph{nearby Lagrangian conjecture}
predicts that $L$ is Hamiltonian isotopic to the zero-section, and hence that the projection map
$\pi\colon L \to M$ is homotopic to a diffeomorphism.

We know from previous work of the first and last author
that $\pi$ is a simple homotopy equivalence. Hence it represents an element of the
\emph{simple structure set} $\SS^s(M)$ considered in surgery theory, namely
the set of such pairs $(L,\pi)$ modulo s-cobordism.  In view of the s-cobordism theorem, when $\dim M \geq 5$ we have that $\pi \colon L \to M$ is homotopic to a diffeomorphism if and only if
the class of $(L,\pi)$ in $\SS^s(M)$ agrees with the class of $(M,\id)$. 

\begin{definition}
We denote by $\SS^s_\NL(M)$ the subset of $\SS^s(M)$ consisting of the classes in $\SS^s(M)$ that admit representatives $(L,\pi)$
where $L$ is a closed manifold and $\pi=\pi_M\circ \psi$ is the composition of the projection map $\pi_M\colon T^*M \to M$ with an exact Lagrangian embedding $\psi\colon L \to T^*M$.
\end{definition}

If the nearby Lagrangian conjecture holds, then it follows in particular that $\SS^s_\NL(M)$ is a singleton consisting of the base point $(M,\id)$. As a step towards the conjecture one may therefore try to show that certain elements of $\SS^s(M)$ cannot arise from nearby Lagrangians.

We observe at this point that with respect to the hard/soft dichotomy in symplectic topology (see Section \ref{sec: hard/soft}), there are already soft obstructions for an element of $\SS^s(M)$ to lie in $\SS^s_\NL(M)$. The most obvious soft obstruction is that $TL \otimes \bC$ and $\pi^* TM \otimes \bC$ should be isomorphic as complex vector bundles, see Section \ref{sec:appli} for further discussion. The focus of the present article is on establishing hard obstructions.

The first hard obstructions on $\SS^s_\NL(M)$ where obtained by the first author \cite{A08} in the case where $M=S^{4k+1}$
and then extended to the case where $M$ is any odd-dimensional sphere by Ekholm-Kragh-Smith \cite{ES16, EKS16}. The combination of these results gives the following: for $k \geq 1$, if $(L,\pi) \in \SS^s_\NL(S^{4k+1})$, then $L$ must bound a parallelizable manifold, and more generally for $k\geq 2$ if $M$ is a $(2k+1)$-dimensional homotopy sphere and $(L,\pi) \in \SS^s_\NL(M)$, then $L = \pm M$ in $\Theta_n/bP_{n+1}$. Here $\Theta_n$ is the group of oriented homotopy spheres up to $h$-cobordism (same as orientation-preserving diffeomorphism in dimension $\geq 5$) with the operation of connected sum and $bP_{n+1} \subset \Theta_n$ is the subgroup consisting of those spheres that bound a parallelizable manifold.


The condition of bounding a parallelizable manifold can be generalized to all manifolds $M$ as follows. A first step in surgery theory is to assign to each element of the simple structure set a map $M\to G/O$ (defined up to homotopy), called
the \emph{normal invariant}, where $G/O$ is the homotopy fiber of the de-looped $J$-homomorphism $BO \to BG$, which associates to a vector bundle the corresponding spherical fibration ($G$ stands for the monoid of stable self-homotopy equivalences of spheres). This gives rise to a (pointed) map
\begin{equation}\label{eq:normal}
\SS^s(M)\to [M,G/O].
\end{equation}
An element of $[M,G/O]$ can be thought of as a vector bundle over $M$ together with a trivialization
of its underlying spherical fibration, up to the appropriate equivalence relation (see Section~\ref{sec:norm-invar-models}). With respect to the operation of direct sum, the map $BO \to BG$ is a map of monoids and indeed of infinite loop spaces, so the homotopy fiber $G/O$ inherits an infinite loop space structure and in particular the set of normal invariants $[M,G/O]$ is an abelian group. When $M=S^n$ the normal invariant of $L\in \SS^s(M)$ is trivial if and only if $L$ bounds a parallelizable manifold, and the group structure on the set of normal invariants coincides with that of $\Theta_n/bP_{n+1}$.

In this article we give an obstruction for realizing a class in $[M,G/O]$ as the normal invariant of a nearby Lagrangian, which yields new constraints on $\SS_{\text{NL}}^s(M)$. The key ingredient is Waldhausen's space of tubes \cite{W82}, which heuristically is the space of codimension zero solid tori in Euclidean space, stabilised with respect to two operations which both raise the dimension of the ambient space, one of which keeps the dimension of the core sphere constant, while the other raises it by one.

We will consider a specific model $\T$ for this space, consisting of quadratic functions of tube type as defined in Section \ref{subsec:twistedder}. Our model has a natural monoid structure and contains the submonoid $\QQ$ of nondegenerate quadratic forms, which serves as our model for $BO$. The space $\T$ classifies generating functions of tube type and the (geometric realisation of the) bar construction $B(\T,\QQ)$ classifies twisted generating functions of tube type, as introduced in \cite{ACGK}. We use $\T^\TOP$, the space of topological tubes, as our model for $BG$, and use $B(\T^\TOP,\QQ)$ as our model for $B(G/O)$. We briefly introduce the key maps in the story.

\begin{enumerate}
\item The {\em classifying map} $M \to B(\T,\QQ)$ associated to a nearby Lagrangian $L \to T^*M$ via the existence theorem for twisted generating functions of tube type for nearby Lagrangians as constructed in \cite{ACGK}.
\item The {\em twisted forgetful map} $B(\T,\QQ) \to B(G/O)$, which is a twisted version of the forgetful map $\T \to BG$, or rather $\T \to \T^\TOP$, for our model $B(\T^\TOP, \QQ)$ of $B(G/O)$.
\item The {\em twisted duality map} $B(G/O) \to G/O$, which is a twisted version of the $S$-duality map $BG \to G$ considered by B\"okstedt and Waldhausen in \cite{BW88}  which we realize as a monoid map $\T^\TOP \to G$ and its twisted version $B(\T^\TOP, \QQ) \to G/O$.
\item The  {\em twisted derivative} $B(\T,\QQ) \to G/O$, which is a twisted version of the  Waldhausen derivative $\T \to G$, which factors as the composition of the maps $\T \to BG$ and $BG \to G$ in items (2) and (3) above \cite{W82,BW88}. 
\end{enumerate}

Our main results are then stated as follows, where equality of maps is understood to mean up to homotopy.

\begin{theorem}\label{thm:main derivative}
Let $M$ be a closed connected manifold and $L$ a closed exact Lagrangian submanifold of $T^*M$.
The normal invariant $M\to G/O$ of the projection $\pi\colon L\to M$ factors as a composition
\[M \to B(\T,\QQ) \to G/O\]
of the classifying map and the twisted derivative.
\end{theorem}

\begin{theorem}\label{thm:main duality}
The twisted derivative factors as a composition
\[ B(\T,\QQ) \to B(G/O) \to  G/O\]
of the twisted forgetful map and the twisted duality map. 
\end{theorem}

\begin{theorem}\label{thm:main torsion}
The twisted duality map $B(G/O) \to G/O$ is 2-torsion.
\end{theorem}

As an immediate application we obtain:

\begin{corollary}\label{cor:main torsion}
Let $M$ be a closed manifold and $L$ a closed exact Lagrangian submanifold of $T^*M$.
The normal invariant $M\to G/O$ of the projection $\pi\colon L\to M$ is a 2-torsion element in the group $[M,G/O]$.
\end{corollary}

\begin{proof} From Theorems \ref{thm:main derivative} and \ref{thm:main duality} we obtain a factoring of the normal invariant of $L \to M$ as the composition 
$$ M \to B(\T,\QQ) \to B(G/O) \to G/O $$ 
The fact that it is 2-torsion then follows from Theorem \ref{thm:main torsion}.
\end{proof}

\begin{example} When $L$ and $M$ are homotopy spheres it follows that $L \# L =  M \# M$ in $\theta_n/bP_{n+1}$. If $n$ is even then $bP_{n+1}=0$ and so we conclude that $L \# L$ and $M \# M$ are in fact diffeomorphic. In particular, if $n$ is even and $\Sigma \subset T^*S^n$ is a Lagrangian homotopy sphere, then it follows that $\Sigma \# \Sigma$ is diffeomorphic to $S^n$.
\end{example}

\begin{example} The factorization through $B(G/O)$ gives additional constraints on the normal invariant beyond the fact that it is $2$-torsion.  For example, $\theta_8 \simeq \theta_8 / {bP}_9 \simeq \bZ/2$ so the fact that the normal invariant of a Lagrangian homotopy sphere $\Sigma \subset T^*S^8$ is 2-torsion gives no new information. However, since  $\pi_7G/O=0$ it follows from our results that any exact Lagrangian homotopy sphere in $T^*S^8$ is diffeomorphic to $S^8$, i.e. the unique exotic possibility is excluded. \end{example}

In fact, for homotopy spheres and more generally for any product of homotopy spheres, we may conclude a stronger result, whose statement will require us to recall some background about the $\eta$ map, which is the map between iterated loop spaces that is induced by the Hopf map.

If $X$ is a 3-fold loopspace $X=\Omega^3Y$ we may define a map $\Omega^2Y \to \Omega^3Y$ by pre-composition with the Hopf fibration $\eta:S^3 \to S^2$. This can also be thought of as a map $BX \to X$ and is referred to as an $\eta$-map. For example $G$ is an infinite loopspace with the operation of direct sum (or rather, smash product of maps), and we have an associated $\eta$-map $BG \to G$.

Waldhausen explained in \cite{W82} that the derivative $\T \to G$ factors as the composition of the forgetful map $\T \to BG$ and a geometrically defined map $BG \to G$. In \cite{BW88} B\"okstedt and Waldhausen identified this map $BG \to G$ with an $S$-duality map $BG \to G$, and then further identified the $S$-duality map with the $\eta$-map $BG \to G$. Hence the derivative $\T \to G$ factors through the $\eta$-map $BG \to G$.

Now, $G/O$ inherits an infinite loop space structure as the homotopy fiber of the $J$-homomorphism $BO \to BG$, which is an infinite loop map under the operation of direct sum, and so there also is an $\eta$-map $B(G/O) \to G/O$. We do not prove in the present article that the twisted duality map $B(G/O) \to G/O$ which appears in Theorems \ref{thm:main duality} and \ref{thm:main torsion} is an $\eta$-map, though this is our expectation. To identify the twisted duality map as an $\eta$-map one would need to keep track of additional structure beyond the monoid structures we consider in this article.



For nearby Lagrangians with stably trivial Gauss map the situation is simpler. As explained in \cite{ACGK}, any nearby Lagrangian $L \subset T^*M$ with stably trivial Gauss map admits a genuine (i.e. untwisted) generating function of tube type, which is classified by a map $M \to \T$. The composition $M \to U/O \to BO$ is trivial, so the normal invariant lifts to a map $M \to G$, and it follows from our discussion below that this map factors as the composition of the classifying map $M \to \T$ and the derivative $\T \to G$. Therefore, from the B\"okstedt-Waldhausen factoring of the derivative as the composition $\T \to BG \to G$ with $BG \to G$ the $\eta$-map, we may directly conclude the following result.

\begin{theorem}\label{cor:main eta1}
Let $M$ be a closed connected manifold and $L$ a closed exact Lagrangian submanifold of $T^*M$. Assume that the stable Gauss map $L \to U/O$ is trivial. Then the normal invariant $M \to G/O$ of the projection $\pi: L \to M$ factors through the $\eta$-map $B(G/O) \to G/O$.
\end{theorem}

It was shown in \cite{ACGK} that the stable Gauss map of a nearby Lagrangian $L \to U/O$ is trivial on homotopy groups, hence nullhomotopic when $M$ (and therefore $L$) has the homotopy type of a sphere. So Theorem \ref{cor:main eta1} applies at least in the case of homotopy spheres. The conclusion may be phrased in more classical terms as follows. We note that this statement was also obtained by Smith and Porcelli in \cite{PS24}  by combining the results of \cite{ACGK} with Floer homotopy theory.

\begin{corollary}\label{cor:main eta2}
Let $M$ be a homotopy sphere and $L$ a closed exact Lagrangian submanifold of $T^*M$. The difference $M-L$ in $\theta_n/bP_{n+1} \subset \text{coker}(J_n)$ is in the image of the composition $$\pi^s_{n-1} \xrightarrow{\eta} \pi^s_n \to \text{coker}(J_n)$$ where the first map is multiplication by the Hopf class $\eta \in \pi^s_1$ and the second map is the quotient by the image of the $J$-homomorphism $J_n:\pi_nO \to \pi_n^s$. 
\end{corollary}



By using known results on the multiplicative structure of the stable homotopy groups of spheres one can use the above consequence to find even stronger restrictions on the possible smooth structures of nearby Lagrangian homotopy spheres.

\begin{remark} In Section \ref{sec:appli} we will also explain how one may obtain concrete applications of our results beyond the case of homotopy spheres,  illustrating the general principle with the specific examples of products of spheres, tori, and complex projective spaces.  \end{remark}

We also mention that in \cite{ACGK} it is proved that the map $M\to B(\T,\QQ)$ also factors
the Gauss map $M\to U/O\simeq B(\Z,\QQ)$. This result is used there in combination with Bökstedt's result that the rigid tube map $BO \to \T$ is a rational
homotopy equivalence (see \cite{B82}) to obtain constraints on the Gauss map.
Note that the factoring $M \to B(\T,\QQ) \to U/O$ already gives information on the composition $M\to G/O \to BO$ and therefore on
the normal invariant $M\to G/O$ itself. Indeed, we have the following homotopy commutative
diagram 
\begin{equation}
\begin{tikzcd}
{B(\T,\QQ)}\arrow[r] \arrow[d] & U/O \arrow[d] \\
G/O \arrow[r] & BO.
\end{tikzcd}
\end{equation}
For instance the map $U/O\to BO$ is $2$-torsion since it is an $\eta$-map with respect to the Bott periodicity equivalence $\Omega (U/O) \simeq \bZ \times BO$, and thus
the composition $M\to G/O\to BO$ is also $2$-torsion. Corollary \ref{cor:main torsion} states that in fact the map $M \to G/O$ was already 2-torsion even before taking the composition with the map $G/O \to BO$.

\begin{remark}\label{rem:next}
The next step in surgery theory (the definition of surgery obstructions in odd and even dimensions)
extends the normal invariant map to an exact sequence of pointed sets involving the simple L-groups associated to $\pi_1 M$:
\[L^s_{n+1}(\Z[\pi_1 M])\to \SS^s(M)\to [M,G/O] \to L^s_n(\Z[\pi_1 M]).\]
This means for instance that if we know that $\pi \colon L \to M$ has trivial normal invariant, then
$\pi$ is bordant to the identity $M\to M$ and the bordism can be surgered to an s-cobordism precisely
if an obstruction in $L^s_{n+1}(\Z[\pi_1 M])$ vanishes.
We do not know how symplectic rigidity may constrain these surgery obstructions. 
\end{remark}

\subsection{Acknowledgments}
The authors thank Noah Porcelli and Ivan Smith for sharing their work at an early stage, Stéphane Guillermou for many stimulating discussions
and John Rognes for comments on a first draft. The second author is also grateful to S\o ren Galatius and Sander Kupers for helpful conversations. 

Part of this work was supported by the Swedish Research Council under grant no. 2021-06594 while the authors were in residence at Institut Mittag-Leffler in Djursholm, Sweden during the spring of 2022.

The second author was partially supported by NSF grant DMS-2203455 and the Simons Foundation.

The third author was partially supported by ANR project COSY (ANR-21-CE40-0002).

The fourth author was partially supported by the VR project ``2018-04237'' titled ``Lagrangian submanifolds, Algebraic K-theory and Quantum Physics'' and the VR project ``2022-04388'' titled ``Floer Homotopy Theory and Twisted Algebraic K-Theory''.

\section{The normal invariant and the space $G/O$} \label{sec:norm-invar-models}

\subsection{Notation and terminology} \label{sec:notation-terminology}
We begin by setting some notation and terminology that will be used in this paper. Given a compact manifold $M$ and a finite totally ordered cover $(M_i)_{i\in I}$, we consider, as Segal did in \cite{S},
the Mayer-Vietoris simplicial space $\MV_\bullet((M_i)_{i\in I})$ whose $k$-simplices
are the spaces $\coprod_{i_0\leq\dots \leq i_k} M_{i_1}\cap \dots \cap M_{i_k}$ and whose structure maps
are given by the various inclusions. The natural projection
$|\MV_\bullet((M_i)_{i\in I})| \to M$ is a homotopy equivalence. We shall use this notion to formulate twisting data:
\begin{definition}[c.f. Appendix A of \cite{ACGK}]
Given a monoid $A$ and a right module $X$ an \emph{$A$-twisted map from $M$ to $X$} is a simplicial map
\[\MV_\bullet((M_i)_{i\in I})\to B_\bullet(X,A)\]
for some cover $((M_i)_{i\in I})$. Two such $A$-twisted maps are equivalent if there is an $A$-twisted map to $X$ on $M\times [0,1]$ which, up to refining the covers, respectively restricts to the given ones on $M\times \{0\}$
and $M\times \{1\}$. We have the following classification result: 
\end{definition}
In the above definition, $  B_\bullet(X,A)$ refers to the simplicial space associated with the bar construction,  whose space of $k$-simplices is $X\times A^k$. We denote $B(X,A)$ its geometric realization and we shall assume throughout that $X$ and $A$ have the homotopy types of CW complexes, and that $A$ is well-pointed:
\begin{proposition}[Proposition A.16 of \cite{ACGK}] \label{prop:homotopy_class_twsited}
There is a natural bijection between
\begin{itemize}
\item Equivalence classes of $A$-twisted maps from $M$ to $X$,
\item Homotopy classes of maps $M\to B(X,A)$.
\end{itemize} \qed
\end{proposition}
 
When $X$ is itself a monoid, and the action on the unit defines a monoid map $A \to X$, we obtain maps of bar constructions
\begin{equation} \label{eq:sequence_bar}
B(X,A)\to BA \to BX,  
\end{equation}
where $BA$ and $BX$ refer to setting the right module to be a point. This bar construction is a model for the homotopy quotient of $X$ by the action of $A$, and in many cases the sequence in Equation~(\ref{eq:sequence_bar}) is a fibration sequence. In particular, $B(X,A)$ is often the homotopy fiber of $BA\to BX$.


Throughout the paper, we shall find it useful to have alternate models of $ B(X,A) $ and similar spaces. To this end, we shall consistently use the following notational convention: let $X_n$ be a collection of spaces indexed by non-negative integers, and equipped with \emph{stabilization} maps $X_n \to X_{n+1}$. We write $X$ for the colimit of $X_n$ given by stabilization, $\mathcal{X}$ for the disjoint union of the spaces $X_n$. We write $\mathcal{X}_\infty$ for the colimit $\mathcal{X}\to \mathcal{X} \to \cdots$ induced by the stabilization maps. In particular $\mathcal{X}_\infty = \Z \times X$.

We warn that two slightly different cases of this will be used. 1) Sometimes each $X_n$ is a monoid and the maps $X_n \to X_{n+1}$ are monoids maps hence $X$ is a monoid. 2) Sometimes $\mathcal{X}$ is the monoid and the product structure respects the $\N$-grading. Moreover, in this case the maps $X_n \to X_{n+1}$ are given by left multiplication by a specific element in $X_1$, and even though one may define an $A_\infty$ monoid structure on $X$ from this data, we will not need it and $X$ does not have an obviously defined monoid structure.

\subsection{Stable vector bundles with spherical trivializations}
Let $M$ be a compact manifold. We write $\id_n$ for the identity map of $\R^n$ (or of $S^n=\R^n\cup\{\infty\}$).

\begin{definition}
A \emph{vector bundle over $M$ with a spherical trivialization} is a pair $(E,\theta)$
where $E$ is a rank $n$ real vector bundle over $M$ and $\theta : J(E) \to M\times S^n$
is a fibered homotopy equivalence (here $J(E)$ is the sphere bundle obtained from $E$
by one-point compactifying each fiber).

Two pairs $(E,\theta)$ and $(E',\theta')$ are \emph{stably equivalent} if
there exist integers $k$ and $k'$ and a vector bundle isomorphism $\psi:E\times \R^k \to E'\times \R^{k'}$
such that $(\theta'\wedge\id_{k'})\circ J(\psi)$ is homotopic to $\theta\wedge \id_k$ through fibered maps.
\end{definition}

Let $O_k$ be the topological group of linear isometries of $\R^k$ and $G_k$ the topological monoid
of self-homotopy equivalences of $S^k=\R^k\cup \{\infty\}=J(\R^k)$ preserving the base point $\infty$. Following the notation set in Section \ref{sec:notation-terminology}, we then let $O=\colim O_k$ and $G=\colim G_k$ under the stabilization map $g\mapsto g\times \id_1$, and
denote $B_\bullet(G,O)$ the simplicial space obtained from the bar construction, and  $B(G,O)$ its geometric realization. Proposition~\ref{lem:bar3} proves that
\begin{align*}
  G \to B(G,O) \to BO
\end{align*}
is a fibration sequence and as $G$ is group like ($\pi_0(G)$ is a group) it is easy to extend this to the right and conclude that $B(G,O) \to BO \to BG$ is a homotopy fibration sequence. This justifies why we denote this space $G/O$ in the introduction. There is, however, a more concrete way of identifying $B(G,O)$ as the homotopy fiber of $BO \to BG$. Indeed, as a null homotopy in $BG$ corresponds to a trivialization we now explain this.

There is a concrete way to realize that $B(G,O)$ is a classifying space for stable vector bundles with a spherical trivialization. We start by noting that an $O$-twisted map from $M$ to $G$ consists of
maps $g_i : M_i \to G$ for each $i$ and $o_{ij} : M_{ij}=M_i\cap M_j \to O$ for each $i<j$
such that $g_i(x) \circ o_{ij}(x)=g_j(x)$ for each $x\in M_{ij}$ with $i<j$ and consequently $o_{ij}(x) \circ o_{jk}(x)=o_{ik}(x)$ for each $x\in M_{ijk}$
for $i<j<k$. From now on we shorten these fiber wise compositions by suppressing the element $x$ of $M$.

Given an $O$-twisted map $(g_i,o_{ij})$ from $M$ to $G$,
we can assume, up to refining the cover (or simply shrinking), that
all $g_i$ and $o_{ij}$ belong to $G_N$ and $O_N$ respectively for some $N\in \N$. This datum allows
us to construct a vector bundle with a spherical trivialization $(E,\theta)$ in the usual way:
set 
\[E=\coprod_i(M_i\times \R^N)/\sim \]
where $(x,v)\in M_j\times \R^N \sim (x,o_{ij}(x)(v))\in M_i\times \R^n$ if $x\in M_{ij}$,
and
\[\theta : J(E) \to M \times S^N\]
by the formula $\theta(x,v)=(x,g_i(x)(v))$ for $(x,v)\in M_i \times \R^N$.

\begin{proposition}
Let $M$ be a compact manifold, the above assignment gives a bijection between
the following sets
\begin{itemize}
\item Equivalence classes of $O$-twisted maps from $M$ to $G$,
\item Homotopy classes of maps $M\to B(G,O)$.
\item Pairs $(E,\theta)$ consisting of a vector bundle and a spherical trivialisation, up to stable isomorphism.
\end{itemize} \qed
\end{proposition}

Let us thus consider $\OO=\coprod_k O_k$ and $\GG=\coprod_k G_k$ and make them topological monoids under direct sum, so that the inclusion $\OO \subset \GG$ is an inclusion of monoids and in particular $\OO$ acts on $\GG$ by using the direct sum. Given an $\OO$-twisted map $(g_i,o_{ij})$ from $M$ to $\GG$ we may define an $O$-twisted map $(g_i',o_{ij}')$ from $M$ to $G$ by
\begin{align*}
  g_i' = g_i \qquad \textrm{and} \qquad o_{ij}' = \id_{n_i} \times o_{ij}.
\end{align*}
This in fact defines a simplicial map 
\begin{align}
  \label{eq:twisted-derivative:1}
  c:B_\bullet(\GG,\OO)\to B_\bullet(G,O)
\end{align}
defined by mapping a $p$-simplex $(g,o_1,\dots,o_p)$ with $g\in G_m$, $o_i\in O_{n_i}$ to

\[(g, \id_m\times o_1,\dots,\id_m\times \id_{n_1}\times \cdots \times \id_{n_{p-1}}\times o_p).\]

Assume that each $g_i$ lands in $G_{n_i}$ for some $n_i \in \N$
and put $N=\max n_i$. Then $g'_i\in O_N$ and $o'_{ij}\in O_N$ and as above
we obtain a vector bundle with a spherical trivialization $(E,\theta)$
associated to $(g'_i,o'_{ij})$.

The map $c$ is in fact a homotopy equivalence, but we will not need this.

 \begin{lemma}\label{lem:minusone}
Let $(g_i,o_{ij})$ be an $\OO$-twisted map from $M$ to $\GG$. Then $(-g_i,-o_{ij})$ represents
the same element of $[M,B(G,O)]$ under the map $c:B(\GG,\OO)\to B(G,O)$.
\end{lemma}
\begin{proof}
It suffices to prove that the vector bundles with spherical trivialization $(E,\theta)$ and $(E^-,\theta^-)$
obtained from the above construction applied to $(g_i,o_{ij})$ and $(-g_i,-o_{ij})$ respectively
are stably equivalent. With the above notations, we observe that the maps
\[ \Phi_i:= \id_{M_i}\times \id_{n_i}\times (-\id_{N-n_i}) : M_i\times \R^N\to M_i\times \R^N\]
satisfy for $x \in M_{ij}$ fixed and $(v,u,w) \in \bR^{n_i} \times \bR^{n_{ij} } \times \bR^{N-n_j}$ that
$$ (\text{id}_{n_i} \times (-o_{ij} ) \times \text{id}_{N-n_j} ) \circ \Phi_i(v,u,w) = 
 (\text{id}_{n_i} \times (-o_{ij} ) \times \text{id}_{N-n_j} )(v,-u-w)$$ $$
= (v,o_{ij}u,-w)
=\Phi_j(v,o_{ij}u,w)
= \Phi_j \circ ( \text{id}_{n_i} \times o_{ij} \times \text{id}_{N-j} )(v,u,w)
$$
So we have a commutative diagram
  \begin{center}
    \begin{tikzcd}
      M_i \cap M_j \times \bR^N  \arrow[r, "\Phi_i"]  \ar[d,"\text{id}_{n_i} \times o_{ij} \times \id_{n_j} "] &  M_i \cap M_j \times \bR^N   \ar[d,"\text{id}_{n_i} \times(- o_{ij}) \times \id_{n_j}"]   \\
       M_i \cap M_j \times \bR^N  \arrow[r, "\Phi_j"]  &  M_i \cap M_j \times \bR^N 
    \end{tikzcd}  
  \end{center}
  and hence an isomorphism of vector bundles $\Phi:E \to E^-$. Further, we have the equality $\theta^-\circ\Phi=-\theta$,
so $(E,-\theta)\sim (E^-,\theta^-)$. If $N$ is even then multiplication by $-1$ is homotopic to the identity map in $O_N$ and thus
$(E,\theta)\sim (E,-\theta)$. If $N$ is odd, then note that $(E,-\theta)\sim (E\oplus \R, (-\theta) \times \id_1)$ by stabilization and 
$ (E\oplus \R, (-\theta) \times \id_1) \sim (E\oplus \R, - ( \theta \times \id_1) )$ by the global isomorphism $\id_E \oplus (-\id_1) : E \oplus \bR \to E \oplus \R$. For $N+1$ even multiplication by $-1$ is homotopic to the identity map in $O_{N+1}$ and thus we also deduce that $(E,\theta)\sim (E,-\theta)$ as desired.
\end{proof}

Given two pairs $(E,\theta)$ and $(E',\theta')$ over $M$ as above we get another one by direct (Whitney) sum
$(E\oplus E', \theta \wedge \theta')$. This endows $[M,B(G,O)]$ with a monoid structure,
which is in fact an abelian group. At the level $N\geq 0$ this is represented by the direct sum (and smash)
\begin{align*}
  B(G_N,O_N) \to B(G_{2N},O_{2N}).
\end{align*}
To describe this using the monoids $\OO$ and $\GG$ one has to be a little careful with reordering coordinates. 
It will be useful for us to have a specific representative for twice a given vector bundle with spherical trivialization. So, let
\begin{align*}
  A_k: \R^k \times \R^k \to \R^{2k}
\end{align*}
be given by 
\[A_k(x,y) = (x_1,y_1,x_2,y_2,\dots,x_k,y_k).\]
Then we may define a (doubling) monoid map $\delta : \GG \to \GG$ by the map
\begin{align*}
  \delta(g) = A_k \circ (g\oplus -g) \circ A_k^{-1}
\end{align*}
for $g\in G_k$. The fact that $A_{k_1}(x,y) \oplus A_{k_2}(x',y') = A_{k_1 + k_2}((x,x'),(y,y'))$ implies that this is indeed a monoid map, as one checks that for $g_1\in G_{n_1}$ and $g_2\in G_{n_2}$:
\begin{align*}
\delta_{n_1+n_2}(g_1 \oplus g_2) &= A_{n_1+n_2}\circ (g_1 \oplus g_2 \oplus -g_1 \oplus -g_2 ) \circ A_{n_1+n_2}^{-1} \\
&= (A_{n_1}\times A_{n_2}) \circ (g_1 \oplus -g_1 \oplus g_2 \oplus -g_2) \circ (A_{n_1} \times A_{n_2} )^{-1} \\
&= \delta_{n_1}(g_1) \oplus \delta_{n_2}(g_2).
\end{align*}
This preserves the submonoid $\OO$, and thus induces a simplicial map
\begin{align*}
  \delta : B(\GG,\OO) \to B(\GG,\OO).
\end{align*}

\begin{lemma}\label{lem:twice}
Let $(g_i,o_{ij})$ be an $\OO$-twisted map from $M$ to $\GG$ representing
$\alpha\in [M,B(G,O)]$, then $(\delta (g_i), \delta (o_{ij}))$ represent $2\alpha \in [M,B(G,O)]$.
\end{lemma}

\begin{proof}

By definition of the group structure on stable isomorphism classes of vector bundles with spherical trivializations, twice $\alpha$ is represented by $(g'_i \oplus g'_i,o'_{ij} \oplus o'_{ij})$. Lemma~\ref{lem:minusone} shows that the vector bundle with spherical trivialization $(E,\theta)$
associated to $(g'_i \oplus (-g_i)',o'_{ij} \oplus (-o_{ij})')$ also represents $2\alpha$, where $(g'_i,o'_{ij})=c(g_i,o_{ij})$ and $((-g_i)',(-o_{ij})')=c(-g_i,-o_{ij})$.

Now, we have 
\[c(\delta(g_i),\delta(o_{ij}))=(A_N\circ (g'_i \oplus (-g_i)')\circ A_N^{-1}, A_N\circ (o'_{ij} \oplus (-o_{ij})') \circ A_N^{-1})\]
 Let $(E',\theta')$ be the corresponding vector bundle with spherical trivialization.  Then $A_N$ gives an isomorphism between  $(E,A_N \circ \theta)$
and $(E', \theta')$. If $N$ is such that $A_N$ is homotopic to the identity we conclude that $(E',\theta') \sim (E,\theta)$ as desired. Otherwise $A_N\times (-\id_1)$
is homotopic to the identity map and we conclude as in the proof of Lemma~\ref{lem:minusone}.\end{proof}

\subsection{The normal invariant of a homotopy equivalence}

Let $L$ and $M$ be closed, connected, smooth manifolds of the same dimension which are homotopy equivalent.

\begin{definition}\label{def:normalinvt}
The \emph{normal invariant} of a homotopy equivalence $\pi: L \to M$ is the element of $[M,B(G,O)]$ which is represented
by a pair $(E,\theta)$ where:
\begin{enumerate}
\item $E \to M$ is a rank $n$ vector bundle (the integer $n$ is independent of the dimension of $L$),
\item $\theta : J(E) \to M\times S^n$ is a fibered map which is smooth near $\theta^{-1}(M\times \{0\})\subset E$
and transverse to $M\times\{0\}$, and such that there exists a diffeomorphism $L\to \theta^{-1}(M\times\{0\})$ lifting $\pi$
(up to homotopy) under the projection $\theta^{-1}(M\times\{0\})\to M$.
\end{enumerate}
\end{definition}

\begin{remark}
Since $\theta^{-1}(M\times\{0\})\to M$ is a homotopy equivalence, for a regular value $x$ of
$\theta^{-1}(M\times\{0\})\to M$, the map $\theta_x : J(E)_x\to S^n$
has degree $\pm 1$ and therefore each $\theta_x$ has degree $\pm 1$. Hence $\theta$ is a fibered homotopy
equivalence and $(E,\theta)$ indeed determines an element of $[M,B(G,O)]$. 
\end{remark}
The following result seems to be well-known, but we are including a proof as we have not been able to identify a reference:
\begin{lemma}
The normal invariant of $\pi$ is well-defined, in the sense that the datum $(E,\theta)$ exists and is unique up to equivalence.  
\end{lemma}
\begin{proof}
To prove existence, we pick, for $k$ large enough, an embedding $j:L\to M\times \R^k$ lifting $\pi$ (up to homotopy)
and denote by $\nu(j)$ its normal bundle. Since $\pi$ is a homotopy equivalence, there exists a vector bundle $F$ of rank $n-k$ over $M$
and a trivialization
\[\pi^* F \oplus \nu(j) \simeq L\times \R^n.\]
Let $E=F \oplus \underline{\R}^k$, the direct sum of $F$ and a trivial rank $k$ vector bundle on $M$. The embedding $j':L\to E$ given by $j'(x)=(0,j(x))$ has trivial normal bundle by construction, so the Pontryagin-Thom construction gives the desired map
$\theta : J(E)\to M\times S^n$.

For uniqueness, we observe that since $L\simeq \theta^{-1}(M\times\{0\})$ has trivial normal bundle in $E$, we have
\[TL\oplus \underline{\R}^n\simeq \pi^*E \oplus \pi^*TM,\]
and since $\pi$ is a homotopy equivalence this relation determines $E$ up to stable vector bundle isomorphism.
Next, by general position, the space of embeddings $L\to E$ lifting $\pi$ is non-empty and connected provided
the rank of $E$ is sufficiently large. So given any two pairs $(E,\theta)$ and $(E',\theta')$ as in
Definition~\ref{def:normalinvt} we can assume (up to equivalence) that $E=E'$ and $\theta^{-1}(M\times\{0\})=(\theta')^{-1}(M\times\{0\})$.

The maps $\theta$ and $\theta'$ induce two trivializations of the normal bundle of $L=\theta^{-1}(M\times\{0\})$ in $E$, which thus
differ by a map $B:L\to O_n$. Now $\theta$ and $\theta'$ are homotopic in the space of fibered maps transverse to $M\times\{0\}$ with
given zero locus if and only if $B$ is homotopic to the constant map to the identity $\id \in O_n$. The latter condition can be achieved by stabilization:
Replace $(E,\theta)$ by $(E \oplus \underline{\R}^n, \theta \oplus \id )$ and  $(E',\theta')$ by $(E'\oplus \underline{\R}^n,\theta'\oplus \id)$. For $s:M\to L$
 a homotopy inverse of $\pi$, the pairs $(E \oplus \underline{\R}^n, \theta \oplus \id )$ and $(E   \oplus \underline{\R}^n, \theta \oplus ( B^{-1}\circ s))$ are related by the bundle isomorphism $\id \oplus (B \circ s): E \oplus \underline{\R}^n \to E \oplus \underline{\R}^n$, hence are equivalent.  So after these stabilizations,
the two trivializations of the normal bundle of $L$ differ by $B\oplus (B^{-1}\circ (s \circ \pi))\simeq B \oplus B^{-1}$ which is
homotopic to the constant automorphism $\id \in O_{2n}$. Hence $(E,\theta)$ and $(E',\theta')$ are stably equivalent.
\end{proof}

Let $(g_i,o_{ij})$ be an $\OO$-twisted map from $M$ to $\GG$ defined on a finite
totally ordered open cover $(M_i)_{i\in I}$ of $M$. If $g_i\colon M_i\times \bR^{n_i}\to \bR^{n_i}$ is smooth near $g_i^{-1}(0)$ and is transverse to the origin, then the zero-sets $g_i^{-1}(0) \subset \bR^{n_i}$ glue together as an abstract manifold. Indeed on $M_{ij}$ there is a canonical diffeomorphism $x \mapsto (x,0)$ between $g_i^{-1}(0)$ and $g_j^{-1}(0)=(g_i \oplus o_{ij})^{-1}(0)$. If $M$ is closed, then the union $\bigcup_i g_i^{-1}(0)$ is then a closed (smooth) manifold over $M$. Assuming that $L$ is a closed manifold equipped with a projection map $\pi$ to $M$, and with a diffeomorphism to $\bigcup_i g_i^{-1}(0)$ through which this projection factors, we conclude:

\begin{proposition}\label{prop:recognizeNI}
The element of $[M, B(\GG,\OO)]$ classifying $(g_i,o_{ij})$ represents the normal invariant
of $\pi$ under the map $c:B(\GG,\OO)\to B(G,O)$.
\end{proposition}

\begin{proof}
Let $(g'_i,o'_{ij})=c(g_i,o_{ij})$ and $(E,\theta)$ the corresponding pair.
The union of the manifolds $g_i^{-1}(0)$ is realized as $\theta^{-1}(M\times\{0\}) \subset E$
and by assumption there is a diffeomorphism $L\to \theta^{-1}(M\times\{0\})$ lifting $\pi$.
So $(E,\theta)$ represents the normal invariant of $\pi$ as in Definition~\ref{def:normalinvt}.
\end{proof}

\section{Twisted generating functions}\label{sec:gf}

Every exact Lagrangian submanifold $L \subset T^*M$ was shown, in \cite{ACGK}, to admit a twisted generating function of tube type. However the model for the space of functions of tube type considered in \cite{ACGK} consists of functions which are linear at infinity, while for constructing a model which is a monoid, it is more convenient to have a model built from functions which are quadratic at infinity. In this section we make the translation between the two models.

\subsection{Orthogonal quadratic forms}

We start with some preliminaries on quadratic forms.

\begin{definition}
Let $Q_n$ denote the space of quadratic forms $q:\bR^n \to \bR$ which are non-degenerate, of eigenvalues $\pm \frac{1}{2}$, and $\QQ=\coprod_n Q_n$ with the monoid structure of direct sum. We call $q \in \QQ$ {\em orthogonal quadratic forms}. 
\end{definition}

\begin{remark}
The choice of eigenvalues $\pm\frac{1}{2}$ is a minor cosmetic choice designed to ensure that, the gradient $\nabla q$ of every element of $q\in Q_n$ lies in $O_n$ (rather than  in $\mathrm{GL}_n(\R)$).
It is irrelevant because the space of all non-degenerate quadratic forms deformation retracts
onto $\QQ$. In \cite{ACGK}, we chose to work with eigenvalues $\pm 1$ since that made a few bounds on twisted generating functions
easier to check but we could have equally chosen eigenvalues $\pm \frac{1}{2}$. Alternatively, one may work with eigenvalues $\pm 1$ by replacing $O_n$ with $2 O_n$ in the present text.
\end{remark}

We have $Q_n=\coprod_{k+l=n} Q_{k,l}$ where $Q_{k,l}$ consists of those $q \in Q_n$ whose  negative eigenspace has dimension $k$. The space $Q_{k,l}$ is diffeomorphic to the Grassmannian of $k$ dimensional real vector subspaces in $\R^{k+l}$. Indeed the identification assigns to a $k$-dimensional vector subspace $F \subset \bR^{k+l}$ the unique $q \in Q_{k,l}$ with $F$ as its negative eigenspace and $F^\perp$ as its positive eigenspace.

\subsection{Twisted generating functions linear at infinity}

\begin{definition}
Let $M$ be a smooth manifold. A \emph{generating function linear at infinity} over $M$ is a smooth function
\begin{equation}
    f:M\times \R_{w} \times \R_{v}^n\to \R
\end{equation}
satisfying the following properties:
\begin{enumerate}
\item $f$ is linear, with respect to the variable $w$, at infinity, in the sense that 
\[f(x;w,v)=w+g(x;v) + \varepsilon(x;w,v), \quad (x;w,v) \in M \times \bR \times \bR^n,\]
where the projection $\text{supp}(\varepsilon) \to M$ is proper, and
\item the fiberwise gradient $\nabla  f^x:M\times \R^{1+n}\to \R^{1+n}$ is transverse to $0$.
\end{enumerate}
\end{definition}
The second condition implies that the fiberwise singular set 
\[\Sigma_f=\{(x;w,v), \nabla f^x (w,v)=0\}\]
is a smooth submanifold and that the map $\Sigma_f \to J^1 M=T^*M \times \R$ given by
$$(x;w,v)\mapsto \left(x, \frac{\partial f}{\partial x}(x;w,v)dx, f(x;w,v)\right)$$
is a \emph{Legendrian immersion}.

Following \cite{ACGK}, we define a stabilization operation $f\oplus_b q$ where $f$ is linear at infinity, $q:M\times \R^m\to \R$ is a fiberwise orthogonal quadratic form and $b:M\to (0,\infty)$ is a smooth function:
\[(f\oplus_b q) (x;w,v,u)=w+g(x;v)+\chi_m(b(x)^{-1} u)\epsilon(x;w,v)+q(u).\]
 This formula involves appropriate cut-off functions $\chi_m:\R^m\to [0,1]$ which satisfy
 \begin{align*}
   \chi_{m+m'}(u,u')& =\chi_m(u)\chi_{m'}(u'), \quad (u,u')\in\R^m\times \R^{m'} \\
\|\nabla \chi_m(u)\| & <\|u\|, \quad u\in \R^m\setminus\{0\}.
 \end{align*}
It is readily checked that $f\oplus_b q$ is still linear at infinity. Moreover if 
\[|\epsilon(x;w,v)|\leq b(x)\]
then
$\Sigma_{f\oplus_b q}=\Sigma_f\times\{0\}$ and the Legendrian immersion obtained from $f$
and $f\oplus_b q$ agree.

Observe now that if $\Sigma_f$ is disjoint from the zero locus of $f$, then $\{ f \leq 0 \} = \bigcup_x\{f^x \leq 0\}$
defines a subbundle of the trivial bundle $M\times \R^{1+n}$. As in \cite{ACGK} we will work with a subclass
where the topology of the fiber of this subbundle is rather simple: each fiber will be obtained from a halfspace by a single
trivial handle attachment.

This is made precise in the next definition where we use an auxiliary function $D:\R\to \R$ of the form $D(w)=w+\epsilon(w)$, with $\epsilon$ compactly supported (and $|\epsilon|\leq 4$),
such that $D$ has two critical points, a negative minimum and a positive maximum (it looks like $w^3-w$ but equals $w$ at infinity).
The choice of such a function is irrelevant as any two such functions are homotopic in this class.
\begin{definition}
A generating function linear at infinity $f$ is of \emph{tube type} if for each $x \in M$,
$f^x$ is transverse to $0$ and there exists $q\in \QQ$ such that $f^x$ is homotopic to $D\oplus_4 q$
among functions which are linear at infinity and are transverse to $0$. 
\end{definition}

\begin{definition}
Let $M$ be a smooth manifold. A \emph{twisted generating function linear at infinity} consists of a directed open cover $(M_i)_{i\in I}$ and data $(n_i,b,f_i,q_{ij})$ relative to this cover, satisfying the following properties:
\begin{enumerate}
\item for all $i$, $f_i : M_i\times \R_w \times \R^{n_i}_{v_i}\to \R$ is a generating function linear at infinity over $M_i$, written $f_i=w+\epsilon_i+g_i$ as before,
\item for all $i<j$, $q_{ij}:M_{ij}\times \R^{n_j-n_i}\to \R$ is a fiberwise orthogonal quadratic form, $b$ is a positive constant, and
for all $x \in M_{ij}=M_i\cap M_j$, we have
\begin{align}
g_j(x;v_j) & =g_i(x;v_i) + q_{ij}(x;v_{ij}) \\
  \epsilon_j(x;w,v_i,v_{ij})& =\chi_{n_j-n_i}(b^{-1}v_{ij})\epsilon_i(x;w,v_i),
\end{align}
where we decompose
\begin{equation}
  \label{eq:1}
  v_j=(v_i,v_{ij}) \in \R^{n_j} = \R^{n_i} \times \R^{n_j-n_i},
\end{equation}
\item for all $i$, $x\in M_i$, $(w,v_i)\in \R^{1+n_i}$, $|\epsilon_i(x;w,v_i)|\leq b$, and
\item for all $i<j<k$, $q_{ij}\oplus q_{jk}=q_{ik}$ over $M_{ijk}=M_i\cap M_j \cap M_k$.
\end{enumerate}
\end{definition}

The last condition is actually implied by the second one but we include it for the sake of clarity.

\begin{remark}\label{rem:legtwisted}
If we had $f_j=f_i\oplus q_{ij}$ then it would be clear that the fiberwise singular sets glue into a manifold
with a Legendrian immersion $\bigcup_i \Sigma_{f_i} \to J^1 M$. The cut-off functions $\chi_{n_j-n_i}$ might ruin this but,
as previously mentioned (and shown in \cite{ACGK}), the bound $b$ ensures that this still holds with the above definition.
\end{remark}

\begin{definition}
A twisted generating function linear at infinity is of \emph{tube type} if each function $f_j$ is of tube type.
\end{definition}

\begin{definition}
Let $L \subset J^1M$ be a Legendrian submanifold. We say that $L$ is generated by a twisted generating
function linear at infinity of tube type $(n_i,b,f_i,q_{ij})$ if the fiberwise singular sets of negative value $\Sigma_{f_j}\cap \{f_j <0\}$  define a Legendrian embedding with image $L$, i.e. on each $M_i$ we have that $f_i|_{\{f_i<0\}}$ is a generating function for $L \cap J^1M_i$ in the usual sense.
\end{definition}

The main result of \cite{ACGK} was the following:

\begin{theorem}\label{thm:acgk}
Let $M$ be a closed connected manifold and $L \subset T^*M$ a closed exact Lagrangian submanifold.
Then $L$ admits a twisted generating function linear at infinity of tube type. \qed
\end{theorem}

For the purpose of the present article it will be more convenient to use a different version of generating functions, which we describe next.

\subsection{Almost quadratic twisted generating functions}
\label{sec:almost-quadr-twist}

\begin{definition} 
 A \emph{fiberwise orthogonal quadratic function} on $M$ is a smooth function
\begin{equation}
    q:M\times \R^n\to \R
\end{equation}
such that for each $x \in M$ the function $q^x(u)=q(x,u)$ is in $Q_n$.
\end{definition}

\begin{definition} \label{def:non-deg-quadratic}
We say that a $C^1$ function $G:\bR^n \to \bR$ is a {\em non-degenerate \quadfunc} if 
\begin{enumerate}
\item $G$ is homogeneous of degree 2, i.e. for all $v\in \R^n$ and $\lambda \geq 0$, $G(\lambda v)=\lambda^2 G(v)$, and
\item $\nabla G(v) \neq 0$ for all $x\in M$ and $v \in \R^n\setminus\{0\}$.
\end{enumerate}
\end{definition}

\begin{definition}\label{def:quadfun}
A \emph{fiberwise non-degenerate \quadfunc} on $M$ is a $C^1$-function $$G:M \times \R^n \to \R$$ such that for each $x \in M$ the function $G^x(u)=G(x,u)$ is a non-degenerate \quadfunc.
\end{definition}

\begin{remark} Beware that we use the wording quadratic function as a shorthand for a function which is merely $\R_+$-homogeneous of degree $2$. A quadratic function need not be a quadratic form. Note that a quadratic form satisfies the second condition of Definition~\ref{def:quadfun} if and only if it is non-degenerate, hence the terminology.  \end{remark}

\begin{remark}
 Since $G_x$ is quadratic, asking that $\nabla G_x(v) \neq 0$ for all $v \neq 0$ is equivalent to asking that the bound $\| \nabla G_x(v) \| \geq a(x) \| v\|$ hold for some continuous function $a:M \to (0,\infty)$. This is also equivalent to asking that $G_x$ be transverse to $0$ on $\R^n \setminus \{ 0 \}$, since $dG_x$ is non-zero along the radial direction away from $G_x^{-1}(0)$.
\end{remark}

\begin{definition}\label{def:almostquadfun}
A \emph{fiberwise almost quadratic function} is a smooth function $$F:M\times \R^n\to \R$$ of the form $F=G+E$
where $G$ is a fiberwise non-degenerate quadratic function and $E:M\times \R^n\to \R$ satisfies
\[\|\nabla E^x\|\leq c(x)\]
for some continuous function $c:M\to (0,+\infty)$ (we may take $c$ to be a positive constant if $M$ is compact).
\end{definition}

\begin{remark} \label{def:quadratic_function_determined_by_almost} The fiberwise non-degenerate quadratic function $G$ is determined by $F$, indeed 
\[G^x(v) = \lim_{ \lambda \to +\infty} \frac{G^x(\lambda v) }{\lambda^2} = \lim_{ \lambda \to +\infty} \frac{ F^x( \lambda v)}{\lambda^2 }.\]
\end{remark}

\begin{definition}
An \emph{almost quadratic generating function} over $M$ is a fiberwise almost quadratic function
$F:M \times \bR^n \to \R$ such that the fiberwise gradient $\nabla F^x:M\times \R^n\to \R^n$ is transverse to $0$.
\end{definition}

\begin{definition}
An almost quadratic generating function is of \emph{tube type} if, 
for all $x \in M$,  its associated non-degenerate quadratic function $G^x$ is homotopic to a non-degenerate quadratic form in the space of non-degenerate quadratic functions.
\end{definition}

\begin{remark}
The generating functions quadratic at infinity that are prominent in generating function theory, are particular cases of tube type almost quadratic generating functions
$F$ where the associated function $G$ is a non-degenerate quadratic \emph{form}.
In Waldhausen's theory (see \cite{W82}), this corresponds to the inclusion of the space $BO_n$ of rigid tubes in the space $T_n$ of all tubes.
\end{remark}

An advantage of the condition of being almost quadratic over being linear at infinity is that it is stable under direct sum with a quadratic form.
This removes the need for cut-off functions in the following definition.

\begin{definition}
An \emph{almost quadratic twisted generating function} over $M$ is given
by a directed open cover $(M_i)_{i\in I}$ of $M$ and data $(n_i,F_i,q_{ij})$ where
\begin{enumerate}
\item for all $i$, $F_i:M_i\times \R^{n_i} \to \R$ is an almost quadratic generating function over $M_i$,
\item for all $i<j$, $q_{ij}:M_{ij}\times \R^{n_j-n_i} \to \R$ are fiberwise orthogonal quadratic forms,
\item for all $i<j$, $F_j=F_i \oplus q_{ij}$ over $M_{ij}$, and
\item for all $i<j<k$, $q_{ij}\oplus q_{jk}=q_{ik}$ over $M_{ijk}$.
\end{enumerate}
\end{definition}

An almost quadratic twisted generating function over $M$ gives rise to a Legendrian immersion in $J^1 M$, see Remark~\ref{rem:legtwisted}.

\begin{definition}
An almost quadratic twisted generating function is of tube type if each $F_i$ is of tube type over $M_i$.
\end{definition}

The next two subsections describe an explicit procedure to transform a (twisted) generating function linear at infinity
into a (twisted) almost quadratic generating function, at the cost of adding one more variable.
In view of Theorem~\ref{thm:acgk}, this will lead to the existence of an almost quadratic twisted generating function
for any nearby Lagrangian (or rather for a Legendrian isotopic to it), see Corollary~\ref{cor:existence-quad-gf}. The reader who is interested in our applications, and is willing to accept this equivalence between linear and quadratic generatic functions, may skip ahead to Section \ref{subsec:twistedder}.

\subsection{From linear to quadratic: untwisted case}

\begin{proposition} \label{prop:quadratic_generating_function}
Let $f$ be a generating function linear at infinity over $M$ for a Legendrian submanifold $L \subset J^1M$. For any open subset $M' \subset M$ with compact closure in $M$, there exists an almost quadratic generating function $F$ over $M'$ which generates
a Legendrian submanifold $L' \subset J^1M'$ which is Legendrian isotopic to $L \cap J^1M'$.
Moreover if $f$ is of tube type then we can assume that $F$ is also of tube type.
\end{proposition}

\begin{remark}
If $M$ is closed we may choose $M'=M$.
\end{remark}

The proof will proceed by an explicit construction, where we first produce from $f$ a fiberwise non-degenerate quadratic function (hence is badly singular at the origin), then smooth this singularity out to obtain the desired almost quadratic generating function. Let $M'$ be an open subset of $M$ and $K$ a compact subset of $M$ such that $M'\subset K$.

As in \cite{ACGK}, we start by expressing the linear at infinity function as a sum
\[f(x;w,v)=w+g(x;v)+\epsilon(x;w,v)\]
where the projection $\supp(\epsilon)\to M$ is proper. 
We fix $c>0$ large enough so that $\supp(\epsilon^x)\subset \{w^2+|v|^2< c^2\}$ for all $x\in K$.

For $a\geq c$ (to be chosen sufficiently large later) we define a function 
$$G_a:M' \times \R^2_{t,w}\times \R_{v}^n\to \R,$$
 which vanishes at the origin, and is given away from it by the formula
\begin{equation} \label{eq:piecewise_quadratic}
    G_a(x;t,w,v)=\begin{cases}
\frac{r^2}{a}\left(\frac{aw}{r}+\epsilon\left(x;\frac{aw}{r},\frac{cv}{r}\right) + g\left(x;\frac{cv}{r}\right)\right)  & \text{ if } 0 \leq t \\
 \frac{r^2}{a}\left(\frac{aw}{r} + g\left(x;\frac{cv}{r}\right)\right)  &\text{ if } t <  0\end{cases}
\end{equation}
where $r=\sqrt{t^2+w^2+|v|^2}$. This formula immediately shows that the restriction of $G_a$ to the fiber over $x \in M'$ is quadratic.

Away from the origin, it is straightforward to check that Expression \eqref{eq:piecewise_quadratic} is smooth across the potential domain of discontinuity, since the condition $a \geq c$ ensures that, whenever $t=0$ and hence $r=\sqrt{w^2+|v|^2} $, the  vector $(\frac{aw}{r},\frac{cv}{r})$ lies outside the ball of radius $c$, so that the term $\epsilon\left(x;\frac{aw}{r},\frac{cv}{r}\right)$ vanishes.

Our next goal is to ensure that the zero-level set of $G_a$ is a smooth manifold away from the origin:
\begin{lemma}\label{lem:propertyFa}
  The restriction of $G^x_a$ to $\R^{2+n}\setminus\{0\}$ is transverse to $0$ if and only if the zero-level set of $f^x$ is transverse to the hypersurface $ \frac{w^2}{a^2}+\frac{v^2}{c^2} = 1$.
\end{lemma}
\begin{proof}
We first prove transversality in the regions $\{0 < t\}$ and $\{t > 0\} $, which can both be equipped with coordinates 
\begin{equation} \label{eq:change_coordinates}
    \left(s=r^2,w'=\frac{aw}{r},v'=\frac{cv}{r}\right)
\end{equation}
which define a diffeomorphism between these regions and the product of $(0,+\infty)$ with the interior of the ellipsoid
\begin{equation}
    (0,+\infty)\times \left\{\frac{w'^2}{a^2}+\frac{v'^2}{c^2} < 1\right\}.
\end{equation}
In these coordinates, the function $G^x_a$ reads simply:
\begin{equation} \label{eq:expressing_F_in_half_space}
  G^x_a(s,w',v') =
  \begin{cases}
    \frac{s}{a}f^x(w',v') &   \textrm{ for } 0 \leq t \\
    \frac{s}{a} \left( w' + g^x(v') \right) &  \textrm{ for } t \leq 0 .
  \end{cases}
\end{equation}
The first expression vanishes transversely since $f^x$ does, and the second because the $w'$ derivative is non-trivial.

It thus remains to ensure transversality along the hypersurface $t=0$, and we can use the closure of either chart for this argument:  in the region $0 \leq t$ the function $G^x_a$ is obtained by pulling back the restriction of $ \frac{s}{a}f^x(w',v')$ to the product
\begin{equation}
    (0,+\infty)\times \left\{\frac{w'^2}{a^2}+\frac{v'^2}{c^2} \leq  1\right\}
\end{equation}
under the smooth map given by Equation \eqref{eq:change_coordinates}. On the boundary, the $t$-derivative of this map identically vanishes, so that the partial derivative $\partial G^x_a / \partial t$ is trivial. This implies that critical points of $G^x_a$ which lie in the hypersurface $t=0$ are exactly the critical points of the restriction of this function. This hypersurface is diffeomorphically identified by Equation \eqref{eq:change_coordinates} with the product of $(0,\infty)$ with the hypesurface $ \frac{w^2}{a^2}+\frac{v^2}{c^2} = 1$, and the expression for $G^x_a$ in Equation \eqref{eq:expressing_F_in_half_space} again identifies the critical points of $G^x_a$ with the product with $(0,\infty) $ of the restriction of the critical points of $f^x(w',v')$ to the boundary of the ellipsoid.
\end{proof}
The set of possible values of the constant $a$ for which the transversality property in the above Lemma holds is a priori disconnected; we distinguish a component as follows:
\begin{corollary}\label{cor: const}
There exists a constant $A>0$ such that the restriction of $G^x_a$ to $\R^{2+n}\setminus\{0\}$ is transverse to $0$ whenever $a>A$.
\end{corollary}
\begin{proof}
  Having assumed that $a$ and $c$ are sufficiently large, the hypersurface $ \frac{w^2}{a^2}+\frac{v^2}{c^2} = 1$ lies in the region where $f^x(w,v) = w + g(v)$. Changing coordinates $(aw'',cv'')=(w,v)$ reduces the problem so showing that the function $w'' + \frac{1}{a} g(cv'')$
  is transverse to the unit sphere whenever the constant $a$ is sufficiently large, but this is clear because this function is $C^1$-close to the coordinate function $w$ on $\{(v'')^2 + (w'')^2 \leq 2\}$.
\end{proof}
We conclude that the function $G^x_a$ has the desired behaviour at infinity. However, note that $G^x_a$ has
no critical points except for a degenerate critical point at the origin. We now proceed to
construct a function $F^x_a$ which equals $G^x_a$ at infinity and will satisfy $\crit(F^x_a)\simeq \{\crit f^x\}\cap\{f^x <0\}$.
More precisely, if $f$ generates a Legendrian $L$ over $M'$, then $F_a$ will generate a Legendrian $L_a$ over $M'$ that is Legendrian isotopic
to $L$ (by an explicit isotopy). For this construction we choose auxiliary smooth functions:
\begin{itemize}
\item $\rho:[0,+\infty)\to \R$ such that $\rho(r)=r$ for $r\geq 1$, $\rho(r)\geq r$, $0\leq \rho'\leq 1$ and $\rho(r)$ is constant for $r$ close to $0$.
\item $h:[0,+\infty)\to \R$ such that $0\leq h'\leq \frac{1}{3}$, $h'=\frac{1}{3}$ on $[1,2]$, $h''<0$ on $(2,3)$ and $h=0$ on $[3,+\infty)$.
\item $\chi:[0,+\infty)\to [0,1]$ such that $\chi=0$ on $[0,1]$, $\chi=1$ on $[2,+\infty)$ and $0\leq \chi'\leq 2$.
\end{itemize}

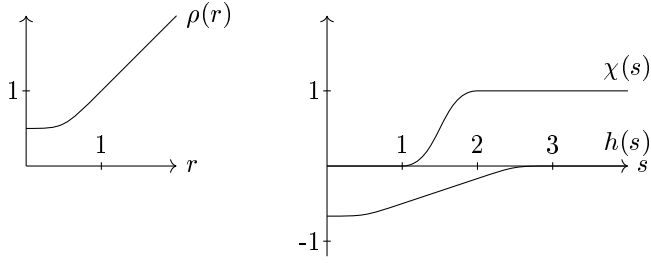
\begin{figure}[h]
  \centering
  \begin{tikzpicture}
     \draw[->] (0,0) -- (2,0) node[right] {$r$};
    \draw[->] (0,0) -- (0,2) node[right] {};
      \foreach \x in {1}
    \draw (\x,-0.05) -- (\x,0.05) node[above] {\x};
  \foreach \y in {1} 
    \draw (-0.05,\y) -- (0.05,\y) node[left] {\y};
     \draw[domain=1:2,smooth,variable=\x] plot ({\x},{\x}) node[right] {$\rho(r)$};
     \draw (0,1/2) ..controls (1/2,1/2) .. (1,1);
   
    \begin{scope}[shift={(4,0)}]
        \draw[->] (0,0) -- (4,0) node[right] {$s$};
    \draw[->] (0,-1.2) -- (0,2) node[right] {};
      \foreach \x in {1,2,3}
    \draw (\x,-0.05) -- (\x,0.05) node[above] {\x};
  \foreach \y in {-1,1} 
    \draw (-0.05,\y) -- (0.05,\y) node[left] {\y};
     \draw (0, -2/3) ..controls (.5,-2/3) ..  (1, -1/2) -- (2,-1/6) ..controls (2.5,0) .. (3,0) -- (4,0) node[above] {$h(s)$} ;
     \draw (0, 0) --  (1, 0) ..controls (1.5,0) and (1.5,1) ..  (2,1) -- (4,1) node[above] {$\chi(s)$} ; 
    \end{scope}
   
  \end{tikzpicture}
  \caption{Three auxiliary functions}
\label{fig:three_functions}
\end{figure}

We define $F_a:M' \times \R^2 \times \R^n \to \R$ by the formula
\[F_a(x;t,w,v)=\frac{\rho^2}{a}\left(\frac{aw}{\rho} + \chi(\rho^2)\epsilon\left(x;\frac{aw}{\rho},\frac{cv}{\rho}\right)\delta_{t>0}+g\left(x;\frac{cv}{\rho}\right)\right)+ h(\rho^2)\]
where $\delta_{t>0}$ is the characteristic function of the half space $t > 0 $ and $\rho$ is the function
\[\rho(r)=\rho\left(\sqrt{t^2+w^2+|v|^2}\right).\]

\begin{lemma} \label{lem:compact_support}
$F^x_a$ is smooth.
\end{lemma}


\begin{proof}
  The only term that is not obviously smooth is $\chi(\rho^2)\epsilon\left(x;\frac{aw}{\rho},\frac{cv}{\rho}\right)\delta_{t>0}$ due to the factor $\delta_{t>0}$. We thus need to check smoothness at points where $t=0$. By definition of $\chi$ the first factor is zero in a neighborhood if $\rho<1$. Hence we restrict to points where $\rho(r)=r \geq 1$. In this case we have $(\frac{aw}{\rho})^2 + (\frac{cv}{\rho})^2 \geq \min(a^2,c^2)=c^2$ hence the middle factor vanishes in a neighborhood (by assumption on $c$).
\end{proof}

We shall again use the coordinate system of Equation \eqref{eq:change_coordinates} in the region $1 \leq r$ and   $0 <t $, where the function is given by
\begin{equation} \label{eq:quadratic_coordinates_G}
  \frac{s}{a}\left(w'+\chi(s)\epsilon^x(w',v')+g^x(v')\right)+h(s).
\end{equation}
If we restrict further to the region $\sqrt{2} \leq r$,
the expression simplifies to
\[\frac{s}{a}f^x(w',v')+h(s). \]
For the statement of the next result, we associate to the function $h$ the diffeomorphism $\varphi_a$ from $(-\frac{a}{3},0) \to (h(2)-\frac{2}{3},0)$ given by
\[\varphi_a(z)=h\left((h'_{\mid [2,3]})^{-1}\left( -\frac{z}{a}\right)\right)+\frac{z}{a}(h'_{\mid [2,3]})^{-1}\left(-\frac{z}{a}\right).\]

\begin{lemma}
  The function $F_a^x$ is almost quadratic and its critical points in the region $\sqrt{2} < r$ and   $0 <t $, whenever the constant $a$ is sufficiently large,  generates  the image of $L$ under the contactomorphism
  \[(x,p,z)\mapsto (x,\varphi_a'(z)p,\varphi_a(z)).\]
\end{lemma}
\begin{proof}
  Since $F_a^x$ agrees with $G_a^x$ outside the ball of radius $\sqrt 3$, it is quadratic. The critical points of $F^x_a$ in the region where $\sqrt 2<r$ and $0<t$ are points $(s,w',v')$ such that $(w',v')$ is a critical point of $f^x$ and $f^x(w',v')=-ah'(s)$.
  As $h'$ is decreasing in $(2,3)$ and defines a diffeomorphism $(2,3)\to (0,\frac{1}{3})$, this will occur precisely
  at one value of $s$ for each critical point of $f^x$ with negative critical value
  (if $a>0$ is large enough so that $-a/3$ is less than all critical values of $f^x$). 
  
  Now, the parameter $s$ is determined by the Equation $\frac{1}{a}f^x(w',v') + h'(s) = 0$, so for a critical point $(w',v')$ with critical value $z$ the corresponding parameter is $s=(h')^{-1}(-z/a)$. This in turn gives a critical value of
$$ \frac{s}{a}f^x(w',v') + h(s) = \frac{sz}{a} + h(s) = \frac{z}{a}(h')^{-1}(-\frac{z}{a} ) + h\left( (h')^{-1}(-\frac{z}{a} ) \right) = \varphi_a(z) .$$

\end{proof}

We now complete the proof of the main result of this section, where we recall
\[F_a(x;t,w,v)=\frac{\rho^2}{a}\left(\frac{aw}{\rho} + \chi(\rho^2)\epsilon\left(x;\frac{aw}{\rho},\frac{cv}{\rho}\right)\delta_{t>0}+g\left(x;\frac{cv}{\rho}\right)\right)+ h(\rho^2)\]

\begin{proof}[Proof of Proposition \ref{prop:quadratic_generating_function}]


By deforming $\varphi_a$ continuously to the identity map $\text{id}:(-\frac{a}{3},0) \to (-\frac{a}{3},0)$, we obtain
a Legendrian isotopy from $L$ to the Legendrian generated by $F_a$, as long as we show that there is no critical point in the ball of radius $\sqrt{2}$ and in the region $t \leq 0$.  

In the region $1 \leq r^2 \leq 2$ and $ 0 < t$, using Expression (\ref{eq:quadratic_coordinates_G}), we compute that 
\begin{align}
  \frac{\del F^x_a}{\del w'} & =\frac{s}{a}\left(1+ \chi(s)\frac{\del\epsilon^x}{\del w'}(w',v')\right).
\end{align}
If $\del F^x_a /\del w'=0$, then $(w',v') \in \supp(\epsilon^x)$ so we have $|(w',v')|\leq c$. As also $1 \leq s \leq 2$ we arrive at a region in which $ \partial F^x_a / \partial s$ is the sum of $h'(s)$ with a function whose $C^1$-norm goes to $0$ in the limit $a \to +\infty$, so the fact that  
$h'(s)=\frac{1}{3}$ implies the non-vanishing of the $s$ derivative 
so that $F^x_a$ has no critical points for large $a$.


In the region $(\{t\leq \delta r\} \cap \{r \leq \sqrt 3\}) \cup \{r\leq 1\}$ with small $\delta >0$, we have
\[F_a(x;t,w,v)=\rho w+\frac{\rho^2}{a} g\left(x;\frac{cv}{\rho}\right)+ h(\rho^2)\]
On this compact set we can for large enough $a$ ignore the middle term and compute:
\[\frac{\del (\rho w + h(\rho^2))}{\del w}=\rho+\rho'(r)\frac{w^2}{r}+2\rho\rho'(r)\frac{w}{r}h'(\rho^2)  \geq (1+2\rho'(r)\frac{w}{r}h'(\rho^2))\rho \geq \frac13 \rho\]
which is bounded from below.

In the region $\{t\leq \delta r\} \cap \{r \geq \sqrt 3\}$ we have
\[F_a(x;t,w,v)=r w+\frac{r^2}{a} g\left(x;\frac{cv}{r}\right) \]
which is homogeneous of degree 2. So the fact that it has no critical points near the boundary $r=\sqrt{3}$ implies that it has no critical points. 



Finally we assume that $f$ is of tube type and prove that the same holds for $F$. By definition, for any fixed $x$, there is a homotopy $f^x_\sigma$, $\sigma \in [0,1]$, from $f^x=f^x_0$
to a function $f^x_1=D\oplus_4 q$ where $q\in \QQ$ and $D: \R\to \R$ is the model function ($D(w)=w+\epsilon(w)$). We drop the superscript $x$ from now on.
The construction above works with the parameter $\sigma$ by choosing the constant $a>0$ large enough (but fixed), and we obtain a family of functions $(F_\sigma, G_\sigma)$, where 
$$G_\sigma(t,w,v) = \frac{r^2}{a} \left( \frac{aw}{r} +  \varepsilon\right(\frac{aw}{r},\frac{cv}{r} \left) \delta_{t >0} + g \left(\frac{cv}{r} \right) \right) $$

We just need to check that $G_1$ is isotopic to a non-degenerate quadratic form.
It writes:
\[G_1(t,w,v)=rw+\frac{r^2}{a}\epsilon(\frac{aw}{r})\delta_{t>0}+ \frac{c^2}{a}q(v).\]
Therefore $G_1$ decomposes as $H\oplus \widehat{q}$
where $\widehat{q}(v)=\frac{c^2}{a}q(v)$ and $H:\R^2\to \R$ is given by
\[H(t,w)=rw+\frac{r^2}{a} \epsilon\left(\frac{aw}{r}\right)\delta_{t>0}.\]
By assumption, the function $D(w)=w+\varepsilon(w)$ vanishes transversely in three points $w_1,w_2,w_3$ (recall $D(w)$ looks like $w^3-w$).
So on the circle $\{t^2+w^2=1\}$ the function $H$ vanishes transversely at $4$ points: 
\[(-1,0),\quad \left(\frac{\sqrt{a^2-w_i^2}}{a}, \frac{w_i}{a}\right)\quad i=1,2,3.\]
So this function is
isotopic to the function $p(t,w)=tw$ through quadratic functions with non-vanishing gradient away from $0$.
Hence there exists an isotopy from $G_1$ to the non-degenerate quadratic form $p \oplus \widehat{q}$ through non-degenerate quadratic functions and hence $G$, and therefore also $F$, is of tube type. 
\end{proof}

\subsection{From linear to quadratic: twisted case}

Let  $(M_i)_{i\in I}$ be a directed open cover of a closed manifold $M$ and let $(n_i,b,f_i,q_{ij})_{i\in I}$ be a twisted generating function linear at infinity with respect to
this cover. This means that  $f_j:M_j\times \R_w \times \R^{n_j}_{v_j}\to \R$ is a smooth function of the form
\[f_j(x;w,v_j)=w+\epsilon_j(x;w,v_j)+g_j(x;v_j),\]
where $q_{ij}(x;-)\in\QQ$ for all $x\in M_{ij}$, and, whenever $i<j$ and $x \in M_{ij}$, we have
\begin{align}
  g_j(x;v_j) & =g_i(x;v_i)\oplus q_{ij}(x;v_{ij}) \\
  \epsilon_j(x;w,v_i,v_{ij})& =\chi_{n_j-n_i}(b^{-1}v_{ij})\epsilon_i(x;w,v_i),
\end{align}
where we decompose
\begin{equation}
  \label{eq:1}
  v_j=(v_i,v_{ij}) \in \R^{n_i} \cong \R^{n_j} \oplus \R^{n_j-n_i}.
\end{equation}
Since $M$ is closed we may assume for simplicity that the indexing set is finite, so we write $i \in \{1,\ldots,k\}$,  and will sometimes find it useful to use the notation $v_1=v_{0,1}$ despite the fact that we do not have any element of the cover labelled by $i=0$.
We assume that for each $j$ and $x\in M_j$, $f_j^x$ is transverse to $0$, and that the map $\supp(\epsilon_j)\to M_j$ is proper for each $j$.

The main result of this section is the extension of Proposition \ref{prop:quadratic_generating_function} to the twisted case.
\begin{proposition}\label{prop:linear_to_quadratic_twisted}
Let $(n_i,b,f_i,q_{ij})$ be a linear at infinity twisted generating function
with respect to an open cover $(M_i)_{i\in I}$ of $M$ by open subsets $M_i \subset M$ with compact closure, which generates
a Legendrian submanifold $L \subset J^1M$ when restricted to $\{f_i<0\}$.
There exists an almost quadratic twisted generating function $(n_i,F_i,q_{ij})$
with respect to an open cover $(N_i)_{i\in I}$ of $M$, with $N_i \subset M_i$, such that the Legendrian submanifold in $J^1M$ which is
generated by the $F_i$ is Legendrian isotopic to $L$. Moreover if $f_i$ is of tube type, then
we can assume that $F_i$ is also of tube type.
\end{proposition}

We start by picking open subsets $M'_i \subset M_i$ with the property that $M=\bigcup_i M'_i$ and $K_i=\overline{M'_i} \subset M_i$,
 and choose cut-off functions $\kappa_i : M\to [0,1]$ such that $\kappa_i$ has support in $M_i$ and
$(\kappa_i)^{-1}(1) \supset K_i$.  We pick $c>0$ large enough so that the support of $\epsilon^x_i$
is contained in $\{w^2+|v_i|^2 < c\}$ for all $i$ and $x\in K_i$.

We first define $G_j : M_j \times \R^2\times \R^{n_j}\to \R$ as in the previous section. This depends on some parameter $a>0$
which is chosen large enough so that Corollary \ref{cor: const} is valid for all $j$ (where the function is defined over the open subset $M_j$, but we only require the bound to work on the compact subset $K_j$).
Recall the explicit formula:
\[G_j(x;t,w,v)= \frac{r_j^2}{a}\left(\frac{aw}{r_j}+\epsilon_j\left(x;\frac{aw}{r_j},\frac{cv_j}{r_j}\right)\delta_{t>0} + g_j\left(x;\frac{cv_j}{r_j}\right)\right) \]
where $r_j=\sqrt{t^2+w^2+|v_j|^2}$.

Then we define maps 
\[\pi_j:M_j \times \R^{n_j}\to \R^{n_j}\]
by the formula:
\[\pi_j(x;v_j)=(v_{0,1},(1-\kappa_1(x))v_{1,2},\dots, \prod_{1 \leq i <  j}(1-\kappa_i(x))v_{j-1,j}).\]
Observe the following key property: for $i<j$, if $x\in M_j \cap K_i$, then $\pi_j(x;v_i,v_{ij})=(\pi_i(x;v_i),0)$ since $\kappa_i(x)=1$.
For convenience we also define
\[\Pi_j:M_j \times \R^2\times \R^{n_j} \to M_j \times \R^2\times \R^{n_j},\]
\[\Pi_j(x;t,w,v_j)=(x;t,w,\pi_j(x;v_j)).\]

We now define a modified version $G'_j$ of $G_j$ using these projections:
\[G'_j=G_j\circ\Pi_j.\]

Observe that for $i<j$, $x\in M_j \cap K_i$ and $(t,w,v_j)\in \R^2_{t,w} \times \R^{n_j}_{v_j}$, we have:
$$
G'_j(x;t,w,v_j)= G'_j(x;t,w,v_i,v_{ij})=G_j(x;t,w,\pi_j(v_i,v_{ij}))$$ $$ =G_j(x;t,w,\pi_i(v_i),0)=G_i(x;t,w,\pi_i(v_i))=G_i'(x;t,w,v_i),
$$
since $q_{ij}(0)=0$ and $\chi_{n-j-n_i}(0)=1$.

Instead of this equality after restriction, we would like to have $G'_j =G'_i \oplus q_{ij}$, so we will compensate for that using the following lemma.

\begin{lemma}
There exists smooth functions $q_i :M_i \times \R^{n_i}\to \R$ which are fiberwise quadratic forms such that
for all $i<j$ and $x\in M_{ij}$, 
\[q_j(x;v_j)=q_i(x;v_i)+q_{ij}(x;v_{ij}).\]
\end{lemma}
\begin{proof} This can be done by induction on the cover since the space of quadratic forms (dropping the non-degeneracy condition) is convex .
\end{proof}

Finally we pick the functions $q_i$ provided by the previous lemma and define:
\[G_j'':M_j\times \R^{2+n_j} \to \R\]
as
\[G_j''=\frac{1}{\alpha} G_j'+q_j\]
for some parameter $\alpha>0$ to be chosen small enough later. The following basic properties are immediate consequences of the construction:
\begin{enumerate}
\item $G_j''$ is a $C^1$-function,
\item $G_j''$ is homogeneous of degree $2$ in the coordinates $(t,w,v_j)$,
\item For $i<j$ and $x\in M'_{ij}$, $G''_j(x;t,w,v_j)=G''_i(x;t,w,v_i)\oplus  q_{ij}(x;v_{ij})$.
\end{enumerate}




\begin{lemma}
There is $\alpha_0>0$ small enough so that for all $0<\alpha<\alpha_0$, all $j$ and $x\in K_j$, $G_j''^x$ is transverse to $0$ away from the origin (or equivalently, when restricted to 
the unit sphere).
\end{lemma}
\begin{proof}
Let us fix $j\in\{1,\dots,k\}$ and $x\in K_j$. Take $i$ minimal such that $x\in K_i$.
Then $\kappa_i(x)=1$ and, for $y$ in a neighborhood of $x$, $\kappa_k(y)\neq 1$ for all $k<i$, so $\Pi_i^y : \R^{2+n_i}\to \R^{2+n_i}$ is a (linear) diffeomorphism.
We know from Lemma~\ref{lem:propertyFa} that, for $y\in K_i$, $G^y_i$ vanishes transversely in $\R^{2+n_i}\setminus \{0\}$ as $a>0$ was chosen big enough, so
the same holds for $G_i'^{y}=G_i^y\circ \Pi_i^y$ with $y$ near $x$. Now $\alpha G_i''=G_i'+\alpha q_i$ restricted to the sphere $\{r_i=1\}$ converges to $G_i'$ in $C^1$-topology
as $\alpha$ goes to $0$, so $\alpha G_i''^{x}$ (and thus $G_i''^{x}$) is also transverse to $0$ in the complement of the origin for $\alpha >0$ small enough. Finally $G_j''^{y}=G_i''^{y}\oplus q^y_{ij}$
is also transverse to $0$ in the complement of the origin since $q_{ij}^y$ is non-degenerate.
As this is valid for $y$ near $x$, by compactness of $K_j$, for $\alpha >0$ small enough
this works for all $j$ and $x\in K_j$.
\end{proof}

Next we define the functions $F_j$ as in the previous section using the auxiliary functions $\rho$, $\chi$, $h$.
Recall $F_j:M_j \times \R^2 \times \R^{n_j} \to \R$ is defined by the formula
\[F_j(x;t,w,v_j)=\frac{\rho_j^2}{a}\left(\frac{aw}{\rho_j} + \chi(\rho_j^2)\epsilon_j\left(x;\frac{aw}{\rho_j},\frac{cv_j}{\rho_j}\right)\delta_{t>0}
+g_j\left(x;\frac{cv_j}{\rho_j}\right)\right)+ h(\rho_j^2)\]
where 
\[\rho_j(t,w,v_j)=\rho(r_j)=\rho\left(\sqrt{t^2+w^2+|v_j|^2}\right).\]

And similarly we define $F'_j:M_j\times \R^{2+n_j}\to \R$ and $F_j'':M'_j\times \R^{2+n_j}\to \R$:
\[F'_j=F_j\circ \Pi_j,\]
\[F_j''=\alpha^{-1}F_j'+q_j.\]
\begin{lemma}
For all $j$, $F_j$ , $F_j'$ and $F''_j$ are smooth.
\end{lemma}
\begin{proof}
  In Lemma~\ref{lem:compact_support} we saw that $F_j$ is smooth. Since $\Pi_j$ and $q_j$ are smooth, it follows that $F_j'$ and $F_j''$ are smooth.
\end{proof}

\begin{lemma}
For $x\in M'_{ij}$, we have $F''_j=F''_i \oplus  q_{ij}$.
\end{lemma}
\begin{proof}
If $x\in M'_{ij}$, then $\Pi_j=(\Pi_i,0)$ and $\rho_i=\rho_j$. 
\end{proof}

\begin{lemma}\label{lem:C1bound}
The difference $F''^x_j-G''^x_j$ is globally $C^1$-bounded uniformly over $x\in M'_j$.
\end{lemma}
\begin{proof}
We have
\[\alpha(F_j''-G_j'')=F_j'-G_j'=(F_j-G_j)\circ \Pi_j.\]
But $F_j-G_j$ is compactly supported on $K_j\times \R^{2+n_j}$ and $|d\Pi_j^x|\leq 1$ for all $x\in M_j$,
so $F_j''^x-G_j''^x$ is $C^1$-bounded uniformly over $x\in M'_j \subset K_j$.
\end{proof}

\begin{proposition}
If the twisted generating function $(M_i,b,f_i,q_{ij})$ transversely generates a Legendrian submanifold $L$ when restricted to $\{f_i<0\}$,
then, for $\alpha>0$ small enough, the twisted generating function $(M_i',F''_i, q_{ij})$ generates transversely a Legendrian submanifold $L''$
which is Legendrian isotopic to $L$.
\end{proposition}
\begin{proof}
Recall that the parameter $a>0$ has been chosen large enough so that each function $F_i$ is a generating
function over $M'_i$ and altogether they generate a Legendrian submanifold $L'=\phi_a(L)$.
Due to the projections $\Pi_i$, the functions $F'_i$ are not generating functions (the transversality hypothesis fails)
but we will prove that $F''_i$ is a generating function and that they altogether generate a Legendrian $L''$
which is $C^1$-close to $L'$ if $\alpha>0$ is small enough.

Fix $j$ and $x\in K_j$. Pick $i$ minimal such that $x\in K_i$. We have $\Pi_j^x=(\Pi_i^x,0)$
and for $y$ near $x$, $\Pi_i^y$ is a (linear) diffeomorphism.
For $y$ near $x$, $F_j''^y=F_i''^y\oplus q^y_{ij}$ so $F_j''$ is a generating function
over a neighborhood of $x$ if and only if $F_i''$ is, and if so they generate the same Legendrian.

Since $G^y_i$ vanishes transversely away from the origin and $\Pi_i^y$ is a diffeomorphism, we have $|\nabla G_i'^y|\geq \beta r_i$ for some $\beta >0$.
From Lemma~\ref{lem:C1bound}, we conclude that $|\nabla F_i'^y|\geq \beta r_i - \gamma$ for some $\gamma >0$.
Since $q_i$ is a quadratic form $|\nabla q^y_i| \leq \delta r_i$ for some $\delta>0$ and therefore 
\[|\alpha\nabla F_i''^y|\geq (\beta -\alpha \delta) r_i - \gamma>0\]
outside of a compact set $C_i\subset \R^{2+n_i}$ for $\alpha$ small enough, and this is all valid over
a neighborhood of $x$. 
Note now that $\alpha F_i''=F'_i+ \alpha q_i$ is arbitrarily $C^2$-close to $F'_i$
on $C_i$ as $\alpha$ goes to $0$. So $F_i''$ is a generating function over a neighborhood of $x$ for small enough $\alpha$.
By compactness of $K_j$, this shows that $F''_j$ is a generating function over $M'_j$
for all $j$ provided $\alpha$ is small enough. Moreover, the Legendrian submanifold generated by the twisted generating function
$(\alpha F''_j, \alpha q_{ij})$ is $C^1$-close to $L'$ as $\alpha$ goes to $0$ and hence is Legendrian isotopic
to $L'$. Finally the Legendrian $L''$ generated by $(F''_j,q_{ij})$ is Legendrian isotopic to the one generated by $(\alpha F''_j,\alpha q_{ij})$,
so $L''$ is Legendrian isotopic to $L'$ and to $L$.
\end{proof}

This completes the proof of Proposition \ref{prop:linear_to_quadratic_twisted}. The upshot of the above discussion is the following:

\begin{corollary}\label{cor:existence-quad-gf}
Let $M$ be a closed connected manifold and let $L \subset T^*M$ be a closed exact Lagrangian submanifold.
Then there exists a Legendrian submanifold $L'$ that is Legendrian isotopic to (any Legendrian lift of) 
$L$ and which admits an almost quadratic twisted generating function of tube type.
\end{corollary}

\begin{proof}
This follows from Theorem \ref{thm:acgk} and Proposition \ref{prop:linear_to_quadratic_twisted}.
\end{proof}

\begin{remark}
In \cite{ACGK}, it is proved that twisted generating functions linear at infinity satisfy the homotopy lifting property, namely they persist under Legendrian isotopy (Chekanov-Sikorav theorem). The same proof works in the almost quadratic setting and so $L$ itself admits an almost quadratic twisted generating function.
We do not expand this part as the weaker corollary above is enough for our purposes since $L'$ and $L$ have the same
normal invariant.
\end{remark}

\subsection{Twisted derivative and normal invariant}\label{subsec:twistedder}

We now return to the constructions of Section \ref{sec:almost-quadr-twist}, which we now interpret as maps to appropriate spaces of functions on $\R^n$. We begin by adding a tube-type condition to Definition \ref{def:non-deg-quadratic}:

\begin{definition} Let $T^q_n$ be the space of \emph{non-degenerate quadratic functions of tube type} on $\R^n$, namely $C^1$-functions $G:\R^n \to \R$ such that:
\begin{enumerate} 
\item $G$ is quadratic (in the weak sense that $G(\lambda v)=\lambda^2G(v)$ for $\lambda \geq 0$, $v\in \R^n$), 
\item $\nabla G (v)\neq 0$  for $v\in\R^n\setminus\{0\}$, and
\item $G$ is homotopic to a non-degenerate quadratic form among such functions.
\end{enumerate}
We topologize $T^q_n$ with the $C^1$-topology (weak and strong are equivalent here due to homogeneity).
\end{definition}

Next, we describe the space of functions which corresponds to Definition \ref{def:almostquadfun}:
\begin{definition}
The space of \emph{almost quadratic functions of tube type} on $\R^n$, denoted $T^{aq}_n$ is the space of pairs $(F,G)$ where $G\in T^q_n$ and $F:\R^n\to \R$ is $C^1$
and $\| \nabla (F-G)(v)\| \leq c$ for some constant $c>0$. 
\end{definition}
As discussed in Remark \ref{def:quadratic_function_determined_by_almost}, the function $G$ is in fact determined by $F$, but we find this description more convenient to describe the topology: we topologize $T^{aq}_n$ with the $C^1$-topology on $G\in T^q_n$, the weak $C^1$-topology on $F$
and in such a way that a constant $c>0$ can be found locally near a pair $(F,G)$.
Note that we have continuous maps $i:T^q_n\to T^{aq}_n$ via $G\mapsto (G,G)$
and $p:T^{aq}_n\to T^q_n$ via $(F,G)\mapsto G$. We have $p\circ i=\id$ and the path $((1-t)F+tG,G)$, $t\in [0,1]$,
defines a homotopy between $\id$ and $i\circ p$. By construction, if $F=G+E:M\times \R^n\to \R$
is an almost quadratic generating function of tube type in the sense of Definition \ref{def:almostquadfun}, then $x\mapsto(F^x,G^x)$ defines a continuous map
$M\to T^{aq}_n$.

We equip $\T^q=\coprod_{n \geq0} T^q_n$ and $\T^{aq}=\coprod_{n \geq0} T^{aq}_n$ with the monoid structure given by direct sum.
The maps $p$ and $i$ are compatible with this monoid structure.

\begin{definition}\label{def: derivative}
The \emph{(Waldhausen)  derivative} is the map $T^{aq}_n \to G_n$ given by $(F,G) \mapsto \nabla F$, where we compactify $\bR^n$ at infinity and use the conditions on $F$ and $G$ to conclude that $\nabla F$ is proper, and the tube condition to see the resulting map is a homotopy equivalence.
\end{definition}

Our model for Waldhausen's rigid tube map is the inclusion $\QQ \to \T^{aq}$ given by $q \mapsto (q,q)$. Note that if $q(v)=\frac{1}{2}(\|v_E\|^2 - \|v_{E^\perp} \|^2)$ for $v=v_E+v_{E^\perp}$
under the orthogonal decomposition $E \oplus E^\perp \subset \bR^n$, then $\nabla q$ is the orthogonal reflection about the vector subspace $E$. This fits into a strictly commutative diagram of monoids

\begin{center}
\begin{tikzcd}
\T^{aq} \arrow[r, "\nabla"] & \GG  \\
\QQ \arrow[u] \arrow[r, "\nabla"] &\OO \arrow[u] . \\
\end{tikzcd}
\end{center}

We therefore may define:

\begin{definition}\label{def: twisted derivative}
The \emph{twisted derivative} is the induced map  $B(\T^{aq}, \QQ) \to B( \GG, \OO)$. 
\end{definition}

\begin{remark}\label{rem:deriv-quad}
As explained above, there is a deformation retraction through monoid maps of $\T^{aq}$ onto $\T^q$, so we may also think of the derivative as the map $\T^q \to \GG$ given by $G \mapsto \nabla G$, and of the  twisted derivative as the induced map $B(\T^q, \QQ) \to B( \GG, \OO)$.
\end{remark}

By the appendix of \cite{ACGK} which compares equivalence classes of twisted maps to homotopy classes of maps to geometric realisations of $2$-sided bar constructions, an almost quadratic twisted generating function of tube type for a Legendrian $L \subset J^1 M$ gives a map $M \to B(\T^{aq}, \QQ)$ which is well defined up to homotopy.

\begin{proposition}\label{prop:normal-deriv}
Let $M$ be a closed connected manifold and let $L \to J^1M$ be a Legendrian embedding which admits an almost quadratic twisted generating function of tube type classified by a map $M \to B(\T^{aq}, \QQ)$ and such that the projection $L \to M$ is a homotopy equivalence. Then the normal invariant of $L\to M$ is represented by
the composition 
\[M \to B(\T^{aq}, \QQ) \xrightarrow{\nabla} B( \GG, \OO) \xrightarrow{c} B(G,O).\]
\end{proposition}

\begin{proof}
  Let $(n_i,F_i,q_{ij})$ be the twisted generating function of tube type representing the map $M \to B(\T^{aq}, \QQ)$ and generating $L$. The twisted derivative map sends this to the $\OO$-twisted map to $\GG$ given by $(\nabla F_i, \nabla q_{ij})$. That $(n_i,F_i,q_{ij})$ generates $L$ precisely means that Proposition~\ref{prop:recognizeNI} applies.
\end{proof}

We can now prove our first main result from the introduction:
\begin{proof}[Proof of Theorem \ref{thm:main derivative}]
Let $L \subset T^*M$ be a nearby Lagrangian. It follows from Corollary \ref{cor:existence-quad-gf} and Proposition \ref{prop:normal-deriv} that the normal invariant of the projection $L \to M$ factors as a map $M \to B(\T^{aq} , \QQ) \to B(\GG, \OO) \to B(G,O)$ where $B(\T^{aq} , \QQ) \to B(\GG, \OO)$ is the twisted derivative. This can be lifted to $B(\T^q,\QQ)$ by Remark \ref{rem:deriv-quad}. So we obtain the desired factoring $M \to B(\T,\QQ) \to B(G,O)$ for either model $\T=\T^{aq}$ or $\T=\T^q$.
\end{proof}

\section{Factorization of the twisted derivative}\label{sec:factorization}

In this section we prove Theorem~\ref{thm:main duality} and Theorem~\ref{thm:main torsion}, namely we
factor the twisted derivative $B(\T,\QQ)\to B(\GG,\OO)\to B(G,O)=G/O$
through a twisted $S$-duality map $B(G/O) \to G/O$, and show that the latter is 2-torsion.
Our model for this map will be given in terms of the space $\T^\TOP$ of topological tubes, which is a model for $BG$.

From now on we implicitly replace all previous spaces (and monoids) with their singular simplicial set and entirely work with simplicial sets. We will nevertheless mostly suppress this from the notation, relying on the homotopy equivalence $|\sing_\bullet X| \simeq X$ and the fact that this equivalence is compatible with taking bar constructions in the sense that there are natural homotopy equivalences
\begin{align*}
  |B(\sing_\bullet X,\sing_\bullet A)| \simeq |B(|\sing_\bullet X|,|\sing_\bullet A|)| \simeq |B(X,A)|
\end{align*}
for a topological monoid $A$ acting on a space $X$ (if both are well-pointed and of CW homotopy type, which all our spaces are). We will even call some simplicial sets spaces as the represent spaces through their geometric realization. By a homotopy equivalence of simplicial sets we will always mean a homotopy equivalence after geometric realization, which in other text is often called a weak equivalence.


\subsection{Topological tubes as a model for $BG$}

Denote by $q_{k,l}$ the quadratic form on $\R^{k+l}$ given by the formula 
\[q_{k,l}(v_-,v_+)=\frac{1}{2}(-\|v_-\|^2+\|v_+\|^2)\]
for $(v_-,v_+)\in \R^k\times \R^l$.
For the next definition, we choose natural numbers $k$ and $l$ and set $n=k+l$.
\begin{definition}\label{def:Ttop}
The \emph{topological tube space  $T^\TOP_{k,l}$} is the simplicial set whose $p$-simplices are continuous functions $f:\Delta^p \times \R^n\to \R$ which satisfy the following properties:
\begin{enumerate}
\item $f$ is fiberwise quadratic: $\forall v \in \R^n, \forall \lambda \geq 0, \forall s\in \Delta^p, f(s,\lambda v)=\lambda^2 f(s,v)$, and
\item there exists an isotopy of homeomorphisms 
\[(\varphi^t)_{t\in[0,1]}:\Delta^p\times \R^n\to \Delta^p\times \R^n\]
such that
\begin{itemize}
\item $\varphi^t(s,v)=(s,\psi^t(s,v))$,
\item $\psi^t(s,\lambda v)=\lambda\psi^t(s,v)$ for all $\lambda \geq 0$,
\item $\varphi^0=\id$,
\item $f(\varphi^1(s,v))=q_{k,l}(v)$.
\end{itemize}
\end{enumerate}
The face and degeneracy maps are defined by the obvious pull backs. 
\end{definition}

The space $T^\TOP_{k,l}$ is by definition connected. We define $T^\TOP_n=\coprod_{k+l=n} T^\TOP_{k,l}$ and $\T^\TOP=\coprod_n T^\TOP_n$
which is evidently a monoid under direct sum, and $\QQ$ (or rather $\sing_\bullet \QQ$) is a sub-monoid. We also define the colimit
\[\T^\TOP_\infty=\colim(\T^\TOP \to \T^\TOP \to \cdots)\]
under the left action by $q_{1,1}(v_-,v_+)=\frac{1}{2}(-v_-^2+v_+^2)$. The connected components of $\T^\TOP_\infty$ are classified by the map 
\[(\sigma,d):\T^\TOP_\infty \to \Z^2\] which projects $T_{k,l}^\TOP$, in place $i\in \N$ of the colimit, to $(k-l,k+l-2i)$.

\begin{proposition}
For all $k,l\in \N$, $T_{k,l}^\TOP$ is a Kan complex. In particular $\T^\TOP_\infty$ is a Kan complex.
\end{proposition}
\begin{proof}
Let $\Lambda \subset \Delta^p$ be a horn and $r:\Delta^p\to \Lambda$ a retraction. Given a horn $(f_s)_{s\in \Lambda}$
in $T_{k,l}^\TOP$ we extend it by $f_s=f_{r(s)}$ for $s\in \Delta^p$.

We need to check the second condition in Definition~\ref{def:Ttop}, namely the existence of a fibered
isotopy $\varphi^t_s$ of homogeneous homeomorphisms with $f_s\circ\varphi^1_s=q_{k,l}$. Let $S_k$ for $k \in \{ 1,\dots,p\}$ denote the non-degenerate $(p-1)$-simplices in the horn so that $\Lambda = \cup_k S_k$.
By definition we have such an isotopy over each $S_k$, but these might not agree on the $p-2$ simplices where a pair of these simplices meet. So, assume for induction that we have already defined the desired isotopy $\Phi^t_s$ for $s\in \cup_{k\leq m} S_k$. Pick a retraction $r_m: S_{m+1} \to S_{m+1} \cap \cup_{k\leq m} S_k$ and an isotopy $\varphi^t_s$ as in the definition over $S_{m+1}$. We then extend $\Phi^t_s$ to any $s\in S_{m+1}$ by
\[\Phi^t_s=\varphi^t_s \circ (\varphi^t_{r_m(s)})^{-1} \circ \Phi^t_{r_m(s)}.\]
This continuously extends the already partially defined $\Phi^t_s$ over $S_{m+1}$ (as it agrees on their intersection where $r_m(s)=s$). We also verify that
\begin{align*}
  f_s(\Phi_s^1)=q_{k,l} \circ (\varphi^1_{r_m(s)})^{-1} \circ \Phi^1_{r_m(s)} = f_{r_m(s)} \circ \Phi^1_{r_m(s)}= q_{k,l}.
\end{align*} 
It follows that $\Phi^t_s$ may be defined over all of $\Lambda$. Finally we observe that the pull back to $\Delta^p$ given by $\Phi^t_{r(-)}$
gives the required isotopy over all of $\Delta^p$, proving that $f_s$ is a $p$-simplex in $T_{k,l}^\TOP$ filling the given horn.

For the last statement, any horn in $\T^\TOP_\infty$ is the image of a horn in some $T_{k,l}^\TOP$ so the claim follows.
\end{proof}

\begin{definition}
Let $E_{k,l}^u$ be the simplicial set whose $p$-simplices are pairs $(f_s,u_s)_{s \in \Delta^p}$ where $(f_s)_{s\in \Delta^p} \in T_{k,l}^\TOP$
and $u_s: \R^k\to \R^n$ is a homogeneous map of degree $1$ which induces a homotopy equivalence $\R^k\setminus \{0\}\to \{f_s<0\}$.
The face and degeneracy maps are defined by the obvious pullbacks.
\end{definition}

Let $E_n=\coprod_{k+l=n} E_{k,l}$ and let $\E=\coprod_n E_n$ which is a monoid under direct sum.
For $k,l\in \N$, we define $u_{k,l}:\R^k\to \R^{k+l}$ by $u_{k,l}(v)=(v,0)$ so that $e_{k,l}=(q_{k,l},u_{k,l})$ is an element of $E_{k,l}$ and set
\[\E_\infty=\colim(\E \to \E \to \cdots)\]
with respect to the left action by $e_{1,1}$.

Let $G_k^u$ be the space of all maps $\R^k\to\R^k$ which are homogeneous of degree $1$ and induce a homotopy equivalence
$\R^k\setminus \{0\}\to \R^k\setminus\{0\}$. It is a monoid under composition of maps, and is homotopy equivalent
to the monoid of \emph{unbased} homotopy equivalences $S^{k-1}\to S^{k-1}$. As a space it sits inside $G_k$ from Section~\ref{sec:norm-invar-models} and their colimits over $k$ are homotopy equivalent.

This monoid $G_k^u$ (really, the simplicial monoid $\sing_\bullet G_k^u$) acts on the right on $E_{k,l}$ by precomposition of the map $u$,
and we denote $B(E_{k,l},G_k^u)$ the associated bar construction.

\begin{proposition} \label{prop:twisted-derivative:2}
The forgetful map $E_{k,l}\to T^\TOP_{k,l}$ is a Kan fibration with fiber homotopy equivalent to $G^u_k$. In particular,  $E_{k,l}$ is a Kan complex. The monoid $G_k^u$ acts on $E_{k,l}$ making it a homotopy principal $G^u_k$-bundle
in the sense that the map $B(E_{k,l},G^u_k) \to T^\TOP_{k,l}$ is a homotopy equivalence.
\end{proposition}

\begin{proof}
Let $(f_s)_{s\in \Delta^p}$ be a simplex in $T^\TOP_{k,l}$ and $(f_s,u_s)_{s\in\Lambda}$ a lift
of some horn $\Lambda\subset \Delta^p$ to $E_{k,l}$. Let $r:\Delta^p\to \Lambda$ be a retraction.
Pick an isotopy $\varphi^t_s$ as in Definition~\ref{def:Ttop} and extend $u_s$ for $s\in \Delta$ by the formula
\[u_s=\varphi^1_s\circ (\varphi^1_{r(s)})^{-1}\circ u_{r(s)}.\]
We check that 
\[f_s\circ u_s=f_s\circ \varphi^1_s\circ (\varphi^1_{r(s)})^{-1}\circ u_{r(s)}=q_{k,l}\circ(\varphi^1_{r(s)})^{-1}\circ u_{r(s)}=f_{r(s)}\circ u_{r(s)}<0\]
and thus $(f_s,u_s)_{s\in \Delta}$ is a p-simplex in $E_{k,l}$.


The fiber of $E_{k,l}\to T^\TOP_{k,l} $ over the basepoint consists of homogeneous maps $\R^k\setminus\{0\}\to\{q_{k,l}< 0\}\simeq \R^k\setminus\{0\}$, hence is homotopy equivalent to $G^u_k$ and the homotopy equivalence may be constructed using the action on a single map.

From the above we deduce that $f:E_{k,l} \to T^\TOP_{k,l}$ is a Kan fibration and for each $q_{k,l} \in E_{k,l}$,
the action of $G^u_k$ on $e_{k,l}$ induces a homotopy equivalence from $G^u_k$ to the fiber $f^{-1}(f(e_{k,l}))$. Corollary~\ref{lem:bar2}
then shows that the induced map $B(E_{k,l},G^u_k) \to T^\TOP_{k,l}$ is a homotopy equivalence.
\end{proof}

Our next task is to show that the space of topological tubes is a model for $BG$, see \cite[Proposition 3.1 and Proposition 5.1]{W82} for closely related arguments. Our proof will rely on the fact that the space $\HH^\TOP(S^{k-1}\times S^{l-1})$ of topological h-cobordisms on $S^{k-1}\times S^{l-1}$ are highly connected. This follows by combining the following results.
  \begin{itemize}
  \item $\HH^\TOP(S^{k-1}\times S^{l-1})\to \colim_m \HH^\TOP(S^{k-1}\times S^{l-1}\times [0,1]^m)$
    is highly connected if $k+l$ is large (the stability theorem for topological $h$-cobordisms on $S^{k-1} \times S^{l-1}$). 
  \item $\colim_m\HH^\TOP(X\times [0,1]^m)\to \colim_m\HH^\TOP(Y\times [0,1]^m)$ is highly connected if $X\to Y$ is highly connected (see \cite{BLR}),
  \item $\HH^\TOP([0,1]^m)$ is contractible (Alexander's trick).
  \end{itemize}  
  
\begin{remark}
The stability theorem for topological h-cobordisms is equivalent to the stability theorem for topological concordances by delooping. For a PL manifold, the stability theorem for topological concordances is equivalent to the stability theorem for PL concordances, indeed the topological and PL concordance spaces coincide, see \cite[Theorem 6.2, p41]{BL74}. Thus for a PL manifold one may deduce the stability theorem for topological concordances from the stability theorem for PL concordances \cite{H75}. Alternatively, one may combine \cite[Theorem C, p451]{BL77}, which implies that for a smooth manifold the stability theorems for smooth and topological concordances are equivalent, with the Igusa stability theorem for smooth concordances \cite{I}. This reasoning does not rely on the PL case, compare with the discussion in \cite[Theorem 1.4.1, p22]{JRW}.
  \end{remark}
  
  The key step in our proof is the following:
\begin{proposition} \label{prop:twisted-derivative:3}
The map $(\sigma,d):\E_\infty \to \Z^2$ is a homotopy equivalence.
\end{proposition}

\begin{proof}
  Note that $\E_\infty$ is a Kan complex. Indeed, $\E_\infty \to \T^\TOP_\infty$ is a Kan fibration as it is the colimit of the Kan fibrations from the lemma above and $\T^\TOP_\infty$ is a Kan complex. It follows that it is enough to show that any simplicial map $\partial \Delta^p \to E_{k,l}$ eventually becomes null homotopic after stabilizing (increasing $k$ and $l$).

  We first observe that by decomposing the stabilization $e_{1,1}^m\oplus (f,u)$ into the composition of $e_{0,m}\oplus (f,u)$ and $e_{m,0}\oplus e_{0,m}\oplus (f,u) \simeq e_{1,1}^m\oplus (f,u)$
  we can assume that $k$ is small compared to $l$ as this is the case for $q_{0,m}\oplus f$ with $m$ large.
  
  Let $(f_s,u_s)_{s\in \del \Delta^p}\in E_{k,l}$ be a family in $E_{k,l}$ that we want to extend to the whole simplex $\Delta^p$.
  By a first homotopy, assuming that $k$ is small compared to $l$, we can make all $u_s : S^{k-1}\to \{f_s <0\}$ smooth embeddings
  and further assume that this family of smooth embeddings is isotopic to the standard inclusion $u_{k,l} : S^{k-1}\to \R^k\times \R^l$.
  By (smooth) isotopy extension, we can thus reduce to the case where $u_s=u_{k,l}$ for all $s\in \del \Delta^p$.
  
  For $\epsilon >0$, set 
  \[q_{k,l}^\epsilon(v_-,v_+)=\frac{1}{2}(-\epsilon v_-^2+v_+^2).\]
  For $\epsilon >0$ small enough,
  we have $\{q_{k,l}^\epsilon \leq 0\} \subset \{f_s <0\}\cup\{0\}$. In particular, the region
  $W_s=\{f_s\leq 0 \leq q^\epsilon_{k,l}\}\cap S^{n-1}$ is a topological h-cobordism on $(q^\epsilon_{k,l})^{-1}(0)\cap S^{n-1}\simeq S^{k-1}\times S^{l-1}$.

  By assuming $k$ and $l$ to be very large compared to $p$, we arrive at the key step, where we trivialize this family of h-cobordisms using the high connectivity of $\HH^\TOP(S^{k-1}\times S^{l-1})$. After isotopy extension (see \cite{EK71}), we 
  find a family of homogeneous isotopies $\psi^t$ of $\R^{k+l}$ such that $f_s\circ \psi^1_s$ coincides
  with $q_{k,l}^\epsilon$ near $(q_{k,l}^\epsilon)^{-1}(0)$. Then a radial isotopy (constant near $q_{k,l}^{-1}(0)$)
  \[\psi^t(v)=\left(t\left(\frac{q^\epsilon_{k,l}(v)}{f\circ\varphi^1(v)}\right)^{\frac{1}{2}}+1-t\right)v\]
  allows us to get $f_s\circ \varphi^1\circ \psi^1_s =q^\epsilon_{k,l}$ everywhere.
  
  This then allows us to extend the family $(f_s,u_s)_{s\in \del \Delta}$
  to the whole simplex $\Delta$ as desired.
\end{proof}

We also define
\[\B(E,G^u)=\coprod_{k,l\in \N} B(E_{k,l},G_k^u)\]
and
\[\B G^u=\coprod_{k\in \N} B G^u_k\]
which are both monoids under direct sum.

Using the left action by $e_{1,1}$ and $u_{1,1}$ we define the colimits

\[\B(E,G^u)_\infty=\colim(\B(E,G^u)\to \B(E,G^u) \to \cdots)\]
and
\[\B G^u_\infty=\colim(\B G^u\to \B G^u\to \cdots).\]

We have two monoid maps
\[\N\times \B G^u=\coprod_{k,l\in \N} B G^u_k \leftarrow \B(E, G^u) \to \T^\TOP\]
which after taking the colimits gives
\begin{equation}\label{eq:TtopBG}
\Z\times \B G^u_\infty \leftarrow \B(E, G^u)_\infty \to \T_\infty^\TOP.
\end{equation}
Here $\B G^u_\infty$ is homotopy equivalent to $\Z\times BG$.

\begin{proposition}\label{prop:BGu}
Both maps in \eqref{eq:TtopBG} are homotopy equivalences. In particular $\T_\infty^\TOP$ is homotopy equivalent to $\Z^2\times BG$.
\end{proposition}

\begin{proof}
The statement about the right arrow follows directly from Proposition~\ref{prop:twisted-derivative:2}. For the left map, we first need to rewrite $\B(E,G^u)_\infty$ in a different way.
We observe that the left action by $e_{1,1}$ on $\E$ is compatible with the left stabilization
$G^u_k\to G^u_{k+1}$ by $g\mapsto \id\times g$. Thus the corresponding colimit
\[ G^u =\colim G^u_k\]
acts on $\E_\infty$ and we have a natural identification
\[\B(E,G^u)_\infty\simeq B(\E_\infty,G^u).\]
Similarly we have
\[\Z\times \B G^u_\infty\simeq \Z\times B(\Z, G^u)=\Z^2\times BG^u,\]
and the claim now follows from Proposition~\ref{prop:twisted-derivative:3}.
\end{proof}

There is an analogous space $F_{k,l}$ consisting of pairs $(q,u)$ where $q\in Q_{k,l}$ and
$u : \R^k \to E^-(q)$ is a linear isometry, which carries a right action of the orthogonal group $O_k$
by precomposing $u$. We have a monoid $\FF=\coprod_{k,l} F_{k,l}$ with a monoid maps
\begin{equation*}
\QQ \leftarrow  \FF \to  \N\times\B O=\coprod_{k,l} BO_k.
\end{equation*}
We also define the colimits $\QQ_\infty$, $\FF_\infty$ and $\B O_\infty$ using
left multiplication by $q_{1,1}$, $e_{1,1}$ and $u_{1,1}$ respectively. We finally consider the monoid
$\B(F,O)=\coprod_{k,l} B(F_{k,l},O_k)$ and the colimit $\B(F,O)_\infty$ (again stabilized on the left).
\begin{proposition}\label{prop:BO}
We have:
\begin{enumerate}
\item The map $\FF_\infty \to \Z^2$ is a homotopy equivalence.
\item The map $\FF_\infty \to \QQ_\infty$ is a fibration with fiber $O=\colim_k O_k$.
\item The maps $\Z\times \B O_\infty \leftarrow\B(F, O)_\infty \rightarrow \QQ_\infty$ are homotopy equivalences.
\end{enumerate}
\end{proposition}
\begin{proof}
The first point is the contractibility of the (stable) Stiefel manifold. The second holds since for each $k,l$, the sequence $O_k \to F_{k,l} \to Q_{k,l}$ is a fibration. 
The third statement follows as in Proposition~\ref{prop:BGu}.
\end{proof}

As already mentioned the space $G/O$ (which we defined as $B(G,O)$) is an infinite loop space in a natural way (for instance via an action of the
linear isometry operad, which is an $E_\infty$-operad) so it makes sense to consider its delooping $B(G/O)$.
However to avoid introducing operadic methods, we give here a concrete model for this delooping. In view of the monoid inclusion 
\[\B O=\coprod_k BO_k \to \B G^u=\coprod_k B G^u_k\]
we have a right action of $\B O$ on $\B G^u_\infty$ and
we define
\begin{equation}\label{eq:B(G/O)}
B(G/O)=B(\B G^u_\infty, \B O).
\end{equation}

\begin{proposition}\label{prop:B(G/O)}
We have a natural homotopy equivalence $B(\T^\TOP_\infty,\QQ) \simeq B(G/O)$.
\end{proposition}
\begin{proof}
Consider the commutative diagram
\begin{center}
\begin{tikzcd}
 BO_k \arrow[d] & \arrow[l] B(F_{k,l},O_k) \arrow[r] \arrow[d] &Q_{k,l}  \arrow[d]\\
 B G^u_k & \arrow[l,] B(E_{k,l}, G^u_k) \arrow[r] &T^\TOP_{k,l}.
\end{tikzcd}
\end{center}
Taking the disjoint union of the top row over all $(k,l) \in \N^2$ we get monoids with respect to direct sum, namely $\N\times \B O$, $\B(F, O)$ and $\QQ$.
Each such monoid acts on the corresponding colimit spaces 
$\Z\times \B G^u_\infty$, $\B(E,G^u)_\infty$ and $\T_\infty^\TOP$ in the lower row. Moreover the maps between these colimit spaces
are homotopy equivalences according to Proposition~\ref{prop:BGu}.

We want to complete the proof by applying Corollary~\ref{lem:bar4}. We first need to check that the monoids
act by homotopy equivalences on the colimit spaces at the bottom. This holds since for 
each of these monoids $M$ and each element $m \in M$ in these monoids there is an $m'\in M$
such that $m\oplus m'$ is homotopic to a power of the stabilizing element $e$ (see \cite[Lemma 4.2]{ACGK}
for the case of $\QQ$). Moreover, according to Proposition~\ref{lem:bar3} the colimit 
\[\mathcal{M}_\infty=\colim(M\xrightarrow{e} M \xrightarrow{e}  \cdots)\]
sits in a fibration sequence
\[\Omega B M \to \mathcal{M}_\infty \to B(\mathcal{M}_\infty,M)\to BM\]
and $B(\mathcal{M}_\infty,M)=\colim B(M,M)$ is contractible according to Proposition~\ref{lem:bar0}.
So the map $\Omega B M\to \mathcal{M}_\infty$ is a homotopy equivalence for each of the 3 monoids. According to Proposition~\ref{prop:BO} these three $\mathcal{M}_\infty$ agree (in a way compatible with the maps between the monoids) hence the induced maps between the three $\Omega BM$ are homotopy equivalences. As each $BM$ is connected this implies that the hypothesis
\[B(G/O)=B(\B G^u_\infty, \B O) \xleftarrow{\sim} B(\B(E,G^u)_\infty,\B(F,O)) \xrightarrow{\sim} B(\T^\TOP_\infty,\QQ).\]
\end{proof}

\subsection{Derivative on topological tubes}

Our goal now is to extend the derivative map $\T \to \GG$ to the bigger space $\T^{\TOP}$
of all topological tubes. To achieve this we first provide tubes with extra data and obtain spaces
$\TT$ and $\TT^{\TOP}$ with forgetful maps $\TT\to \T$, $\TT^{\TOP}\to \T^{\TOP}$ which are
homotopy equivalences. Then we will define a topological derivative map
\[\tnabla :\TT^{\TOP} \to \GG\]
and prove that the following diagram is homotopy commutative through monoid maps
\begin{equation}
\begin{tikzcd}
\TT \arrow[r]\arrow[d]& \TT^{\TOP}\arrow[d,"\tnabla"]\\
\T \arrow[r,"\nabla"] & \GG.
\end{tikzcd}
\end{equation}

\begin{definition}
Let $f\in T_n^{\TOP}$. An \emph{adapted pair of deformation retractions} is a pair $(d^-,d^+)$
of continuous maps $[0,1]\times \R^n\to \R^n$ which satisfy:
\begin{enumerate}
\item $d_0^\pm = \id$,
\item $\forall \lambda \geq 0, \forall s\in[0,1], \forall v\in \R^n$, $d_s^\pm(\lambda v)=\lambda d_s^\pm(v)$,
\item $d_1^\pm(\R^n)\subset \{\pm f \geq 0\}$,
\item $d_1^\pm(\{\pm f\geq 0\}\cap S^{n-1}) \subset \{\pm f >0\}$,
\item for all $v\in \R^n$, $s\mapsto \pm f\circ d_s^\pm(v)$ is non-decreasing.
\end{enumerate}
\end{definition}

\begin{definition}
We define $\widetilde{T}_n^{\TOP}$ as the simplicial set whose vertices consist of triples $(f,d^-,d^+)$ where $f\in T_n^{\TOP}$
and $(d^-,d^+)$ is an adapted pair of deformation retractions, and $p$-simplices consist of families of such parametrized by $\Delta^p$ such that the combined maps $\Delta^p \times [0,1] \times \bR^n \to \bR^n$ are continuous. Face and degeneracy maps are defined in the obvious way. We also define the monoid $\TT^\TOP=\coprod_n\widetilde{T}_n^\TOP$.
\end{definition}

\begin{lemma}\label{lem:adapted}
The forgetful map $\TT^\TOP \to \T^\TOP$ is a Kan fibration and a homotopy equivalence.
\end{lemma}
\begin{proof}
We only prove the path lifting property, the case of a general horn being similar.

By isotopy extension any family $f_t\in T_n^\TOP$, for $t\in [0,1]$ can be written as $f_t=f_0\circ\varphi_t$
for some homogeneous isotopy $\varphi_t:\R^n\to \R^n$. Hence an adapted pair $d^\pm$ for $f_0$
can be carried along as $\varphi_t^{-1}\circ d_s^\pm$.

Consider the fiber over some function $f \in \T^\TOP$ and a family $(d_{s,t}^\pm)_{t\in \del \Delta}$
for some simplex $\Delta$. We focus on $d_s^-$ since the situation is completely symmetric.
There exists $\epsilon >0$ such that $d_{1,t}^-(S^{n-1}) \subset \{f\leq \epsilon r^2\}$
for all $t\in \del \Delta$.
We can construct a specific deformation retraction $D^-_s$ which is the identity map on $\{f\leq \epsilon\}\cap S^{n-1}$.
We then deform the family $(d_{s,t}^-)_{t\in \del \Delta}$ to $D_s^-$ by the following explicit formula:
\[\delta^-_{s,t,u}= \begin{cases} d^-_{s,t},& 0\leq s \leq u \\
D^-_{\frac{s-u}{1-u}}\circ d^-_{u,t},& u \leq s \leq 1\end{cases}\]
This satisfies $\delta^-_{s,t,0}=D^-_s$ and $\delta^-_{s,t,1}=d^-_{s,t}$. It is straightforward to check that $\delta^-_{s,t,u}$ satisfies all properties of an adapted deformation retraction for $f$.
\end{proof}

For $q\in Q_n$, we can define a canonical adapted pair $(d_s^+,d_s^-)$ as follows:
write every element $v\in \R^n$ as the sum $v=v_-+v_+$ with respect to the decomposition into positive and negative eigenspaces and set
\[d_s^-(v)=v_-+(1-s)v_+,\quad d_s^+(v)=(1-s)v_-+v_+.\]
This defines an inclusion map
\[\QQ\to \TT^\TOP.\]
The map $v\mapsto d_1^+(v)-d_1^-(v)=v_+-v_-$ is the orthogonal reflection and coincides with the derivative $\nabla q$.

Now for $(f,d^-,d^+)\in \TT^\TOP$ we observe that:
\[\forall v\in \R^n\setminus \{0\},\quad d_1^+(v)- d_1^-(v)\neq 0.\]
Indeed suppose $v\in \R^n\setminus\{0\}$ satisfies $d_1^+(v)=d_1^-(v)$. If $f(v)\geq 0$,
then $f\circ d_1^+(v)> f(v)$ but $f\circ d_1^+(v)=f\circ d_1^-(v)\leq f(v)$ which is a contradiction,
and the case $f(v)\leq 0$ is impossible too.

\begin{definition}\label{dfn:dualitymaps}
  We define the unstable \emph{$S$-duality map} $\tnabla : \TT^\TOP\to \GG$ as the map
  \[\tnabla(f,d^\pm) (v)=d_1^+(v)-d_1^-(v).\]
  It is a monoid map and thus induces the \emph{unstable twisted $S$-duality map}
  \[\tnabla : B(\TT^\TOP,\QQ)\to B(\GG, \OO).\]
\end{definition}
This indeed takes values in $\GG$ since being in $\GG$ is a homotopy invariant condition and
any $(f,d_s^\pm)\in \TT^\TOP$ is homotopic to some $q\in \QQ$ by definition of $\T^\TOP$ and Lemma~\ref{lem:adapted}.

We now want to compare $\nabla$ and $\tnabla$ on smooth tubes.

\begin{definition}
Let $\widetilde{T}_n$ be the simplicial set whose vertices consist of quadruples $(f,d^-,d^+,X)$ where
$(d^-,d^+)$ is an adapted pair for $f\in \T_{k,l}^q$ with $k+l=n$ and $X$ is a continuous vector field on $\R^n$ satisfying
\begin{itemize}
\item $X(\lambda v)=\lambda X(v)$ for all $\lambda\geq 0$ and $v\in \R^n$,
\item $d_v f(X(v))>0$ for all $v\in \R^n\setminus\{0\}$,
\item $f(v+tX(v))>f(v)$ for all $t\in [0,1]$ and $v\neq 0$.
\end{itemize}
and $p$-simplices consist of $\Delta^p$-families of such, continuous in the parameter direction,
with face and degeneracy maps defined in the obvious way.
\end{definition}

We let $\TT=\coprod_n \widetilde{T}_n$ and make it a monoid under direct sum of all data.

\begin{lemma}\label{lem:forgetilde}
The forgetful map $\TT \to \T^q$ is a Kan fibration and a homotopy equivalence.
\end{lemma}
\begin{proof}
The proof is almost the same as for $\TT^\TOP \to \T^\TOP$ (using isotopy extension in the smooth case),
the only difference is the vector field $X$. We claim that this space of vector fields for a given $f$ is contractible.
The first and second condition defines a convex subset of the space of all vector fields.
The second condition implies the third for $t\in[0,\epsilon]$ and $\epsilon>0$ small enough.
Hence after shrinking $X$ to $\epsilon X$, the third property is valid and the result follows. 
\end{proof}

\begin{proposition}\label{prop:derivativesagree}
The maps $$\TT \to \TT^\TOP \xrightarrow{\tnabla}  \GG \qquad \text{and} \qquad \TT\to\T\xrightarrow{\nabla} \GG$$  are homotopic through monoid maps.
\end{proposition}
\begin{proof}
For $(f,d^+,d^-,X) \in \TT$, we define a homotopy between $d_1^+(v)-d_1^-(v)$ and $\nabla f(v)$
in two steps:
\begin{enumerate}
\item $d_{1-s}^+(v+sX(v))-d_{1-s}^-(v)$ for  $s\in[0,1]$, and
\item $(1-s)X(s)+s\nabla f(v)$ for $s\in [0,1]$.
\end{enumerate}
We only need to check that the expression is non-vanishing for $v\neq 0$ and hence defines a homotopy in $\GG$.
The second step is clear since $df(X)>0$ and $df(\nabla f)>0$ away from $0$. For the first one we use the function $f$: for $s\in(0,1]$, 
\[f(d_{1-s}^+(v+sX(v)))\geq f(v+sX(v))>f(v)\geq f(d_{1-s}^-(v))\]
and thus $d_{1-s}^+(v+sX(v))\neq d_{1-s}^-(v)$.
Observe that this is compatible with the direct sum so we obtain the desired homotopy through monoid maps.
\end{proof}

We define the colimit
$$\GG_\infty = \colim \left( \GG \to \GG \to \cdots \right) $$
using direct sum on the left by $\nabla q_{1,1}=\tnabla q_{1,1}$.
\begin{lemma}\label{lem:colimG/O}
  The map $c: B(\GG,\OO) \to B(G,O)$ from Equation~(\ref{eq:twisted-derivative:1}) factors up to homotopy through the natural map $B(\GG,\OO) \to B(\GG_\infty,\OO)$. 
\end{lemma}

\begin{proof}
  This follows by realizing that the diagram
  \begin{center}
    \begin{tikzcd}
      B(\GG,\OO) \arrow[r,"\nabla q_{1,1} \oplus \cdot "] \arrow[dr,"c"] & B(\GG,\OO) \arrow[d,"c"] \\
      & B(G,O)
    \end{tikzcd}
  \end{center}
  homotopy commutes. Indeed, the stabilization does not change the stable equivalence class of the associated vector bundle with trivialized spherical fibration. 
\end{proof}

Using Lemma~\ref{lem:forgetilde}, Proposition~\ref{prop:B(G/O)} and Lemma~\ref{lem:colimG/O}, we can define
(up to homotopy) the \emph{twisted duality map} which we refered to in Section~\ref{sec:intro}.

\begin{equation}\label{eq:twistedduality}
    B(G/O) \simeq B(\TT_\infty^\TOP,\QQ) \xrightarrow{\tnabla} B(\GG_\infty,\OO) \xrightarrow{c} B(G,O)=G/O.
\end{equation}

We now prove our second main result:

\begin{proof}[Proof of Theorem \ref{thm:main duality} ]
  By Proposition \ref{prop:derivativesagree}, we have a homotopy commutative diagram:
  
  \begin{center}
    \begin{tikzcd}
      B(\TT,\QQ) \arrow[r]\arrow[d] & B(\TT^\TOP,\QQ) \arrow[d,"\tnabla"] \\
      B(\T,\QQ) \arrow[r,"\nabla"] & B(\GG,\OO).
    \end{tikzcd}
  \end{center}
  The left vertical map is a homotopy equivalence in view of Lemma~\ref{lem:forgetilde}. By taking the colimit of the right vertical map and using the lemma above we get a factorization of $\nabla : B(\T,\QQ) \to B(\GG,\OO) \xrightarrow{c} B(G,O)$ as
  \begin{align*}
    B(\T,\QQ) \simeq B(\TT,\QQ) \to B(\TT_\infty^\TOP,\QQ) \to B(\GG_\infty,\OO) \to B(G,O).
  \end{align*}
  The composition of the last two maps is what we defined as the twisted duality map.
\end{proof}

\subsection{2-torsionness}
Here we prove Theorem~\ref{thm:main torsion} which states that the twisted duality map \eqref{eq:twistedduality} is $2$-torsion.

Consider the monoid map $\delta : \T^\TOP \to \T^\TOP$ which assigns to a tube $f\in T_n^\TOP$ the tube $\delta f \in T_{2n}^\TOP$ of twice the dimension
defined by 
\[\delta f=(f\oplus (-f))\circ A_n^{-1},\] namely
\[\delta f (x_1,y_1,\dots,x_n,y_n)=f(x_1,\dots,x_n)-f(y_1,\dots,y_n).\]
The map $\delta$ is similarly defined on $\TT^\TOP$ for the extra structure $(d^-,d^+)$. It also preserves the submonoid $\QQ$ (with or without the extra $(d^-,d^+)$).

It does not commute with the left stabilizations by $q_{1,1}$ (with the canonical $(d^-,d^+)$). We therefore define
\begin{align*}
  \TT^\TOP_{\delta\infty} = \colim(\TT^\TOP \to \TT^\TOP \to \cdots)
\end{align*}
using $\delta(q_{1,1})$ instead of $q_{1,1}$. We have not defined $\delta$ on $\E$ nor will we. However, we do define $\delta(e_{1,1}) = (\delta(q_{1,1}),u')$ where $u'(x_1,y_2) = (x_1,0,0,y_2)$, and use stabilization by this element to define
\begin{align*}
  \E_{\delta \infty} = \colim(\E \to \E \to \cdots).
\end{align*}
As stabilization by $\delta(q_{1,1})$ is homotopic to stabilization by $q_{1,1} \oplus q_{1,1}$, it follows from Proposition~\ref{prop:twisted-derivative:3} that 
we also have
\[\E_{\delta \infty} \simeq \FF_{\delta \infty} \simeq \Z^2.\]

\begin{lemma}\label{lem:deltanull}
  The map $\delta : \TT_\infty^\TOP\to \TT_{\delta \infty}^\TOP$ is nullhomotopic, and so is the induced map $\delta :B(\TT_\infty^\TOP,\QQ) \to \B(\TT_{\delta \infty}^\TOP,\QQ)$.
\end{lemma}

\begin{proof}
In view of Lemma~\ref{lem:forgetilde}, it is enough to show that the composition of $\delta$ with the forgetful map $\TT_{\delta \infty}^\TOP \to \T_{\delta \infty}^\TOP$ is nullhomotopic.
For $(f,d^-,d^+)\in \widetilde{T}_n^\TOP$, we consider the map $u:\R^n\setminus\{0\}\to \{\delta f < 0\}$ defined by
\[u(v)=A_n(d_1^+(v),d_1^-(v)).\]
It is well defined since $f(d_1^+(v))>f(d_1^-(v))$ for all $v\neq 0$ as we saw above.
For $f\in \QQ$ this map is a homotopy equivalence and thus it is a homotopy equivalence for any $f\in \TT^\TOP$.
Hence this map $u$ provides a canonical lift of $\delta$ to the highly connected space $E_{n,n}$.

Moreover we observe that this lift to $\E$ is compatible with direct sums.
For $f\in \QQ$, the exact same formula provides a lift of $\delta : \QQ\to \QQ$ through $\FF$, so we
obtain the following commutative diagram of monoids:

\begin{center}
  \begin{tikzcd}
    \QQ \arrow[r]\arrow[d] & \FF \arrow[r]\arrow[d]& \QQ\arrow[d] \\
    \TT^\TOP \arrow[r] &\E \arrow[r]& \T^\TOP
  \end{tikzcd}
\end{center}
where the horizontal compositions are the maps $\delta$. In particular we obtain a factorization of $\delta$ on the quotient:
\[B(\TT_\infty^\TOP,\QQ)\to B( \E_{\delta \infty},\FF)\to B( \T_{\delta \infty}^\TOP,\QQ)\]
and the second statement follows since 
\[B( \E_{\delta \infty},\FF)\simeq B( \FF_{\delta \infty},\FF) = B(\colim \FF,\FF)\simeq\colim B(\FF,\FF)\] is a contractible space.
\end{proof}

We now complete the proof of the last of the main results from the introduction:

\begin{proof}[Proof of Theorem \ref{thm:main torsion}]
  We define
  $$\GG_{\delta \infty} = \colim \left( \GG \to \GG \to \cdots \right) $$
  using the left action by $\nabla \delta(q_{1,1})$. Then we have a commuting diagram
  \begin{center}
    \begin{tikzcd}
      B(\TT^\TOP_\infty,\QQ) \arrow[r,"\tilde \nabla"] \arrow[d, "\delta"] & B(\GG_\infty, \OO) \arrow[d, "\delta"] \\
      B(\TT^\TOP_{\delta \infty},\QQ) \arrow[r,"\tilde \nabla"]  & B(\GG_{\delta \infty}, \OO)
    \end{tikzcd}  
  \end{center}
  where the right vertical $\delta$ is the one defined in Section~\ref{sec:norm-invar-models}. Again the map $c: B(\GG_{\delta \infty} ,\OO) \to B(G,O)$ is well defined as stabilization does not change the stable equivalence class of the vector bundle with spherical trivialization. One way around the diagram represents, by Lemma~\ref{lem:twice}, twice the twisted duality map. The other way around is null homotopic by Lemma~\ref{lem:deltanull}.
\end{proof}

\section{Applications}\label{sec:appli}

In this final section we extract some concrete consequences of our main results by comparing with the surgery theory literature. Our applications will be formulated in terms of the
\emph{simple structure set} $\SS^s(M)$ defined in the introduction, where we recall that $\SS^s_\NL(M)$ denotes the subset of $\SS^s(M)$ consisting of (equivalence classes of) pairs $(L,\pi)$
where $L$ is a closed manifold and $\pi=\pi_M\circ \psi$ for $\pi_M\colon T^*M \to M$ the
projection and $\psi\colon L \to T^*M$ an exact Lagrangian embedding.

\subsection{Hard versus soft obstructions}\label{sec: hard/soft}

If the nearby Lagrangian conjecture holds then $\SS^s_\NL(M)$ is reduced to the base point $(M,\id)$. We will apply our results to exhibit concrete obstructions for an element of $\SS^s(M)$ to lie in $\SS^s_\NL(M)$. Moreover we will exhibit {\em hard} obstructions, with respect to the hard/soft dichotomy, which we briefly recall.

Given a geometric problem, for example the problem of constructing an immersion of a smooth manifold $X$ into another smooth manifold $Y$, one may formulate an underlying {\em formal} problem which often (but not always) is obtained by decoupling the function under consideration from its derivatives. For example, the formal problem underlying the geometric problem of constructing an immersion $X \hookrightarrow Y$ is the problem of constructing a bundle monomorphism $TX \to TY$. Following Gromov, when any solution to the formal problem may be deformed to a solution to the original geometric problem, one says that there holds an h-principle, or that the problem is flexible. For example, the problem of constructing an immersion is flexible when $\dim X < \dim Y$, this is the Hirsch-Smale-Whitney immersion theory.

Although a surprising number of problems in symplectic and contact geometry abide by an h-principle, many do not, and the problem of constructing Lagrangian embeddings is certainly not flexible. For problems which are not flexible, one may still consider the underlying formal problem, and any obstruction to the solution of the problem which is already present at the formal level is said to be {\em soft}. In contrast, any obstruction to the solution of the problem which is not present at the formal level is said to be {\em hard}. We refer the reader to \cite{EM02} for further discussion of the hard/soft dichotomy.

In this section we will apply our results to give new obstructions for an element of $\SS^s(M)$ to lie in $\SS^s_\NL(M)$. Before doing so, we will discuss soft obstructions to this problem and state readily verifiable sufficient conditions for the absence of soft obstructions. When our main results are applied to obstruct an element of $\SS^s(M)$ from lying in $\SS^s_\NL(M)$ in a situation where these sufficient conditions hold, we will know that the obstruction is hard. 

A first obstruction to the problem of realizing an element $[\pi:L \to M]$ of $\SS^s(M)$ in $\SS^s_\NL(M)$ is whether $\pi$ factors (up to homotopy) as the composition of a Lagrangian immersion $L \to T^*M$ and the projection $\pi_M:T^*M \to M$. The problem of constructing Lagrangian immersions is fully flexible: it abides by an h-principle, and indeed reduces to asking that the complex vector bundle $\pi^* (TM \otimes \C)$ is isomorphic to $TL \otimes \C$. In a different direction, one may forget about symplectic geometry and consider the obstruction to realizing an element $[\pi:L \to M]$ of $\SS^s(M)$ in $\SS^s_\NL(M)$ given by factoring $\pi$ (up to homotopy) as the composition of a smooth embedding $L \to T^*M$ and the projection $\pi_M:T^*M \to M$. These two obstructions may be combined into a single obstruction, namely that of factoring $\pi$ (up to homotopy) as the composition of a totally real embedding $L \to T^*M$ and the projection $\pi_M:T^*M \to M$, which is indeed a standard formulation of the formal problem underlying the construction of a Lagrangian embedding. The situation can be summarized in the following result. 

\begin{proposition}\label{prop:soft}
If $(L,\pi)\in \SS^s_\NL(M)$, then the following conditions are satisftied
\begin{enumerate}
\item there exists a totally real embedding $L\to T^* M$ homotopic to $\pi$,
\item there exists an immersion $L \to T^* M$ homotopic to $\pi$ which is both regularly homotopic to a Lagrangian immersion and
to a smooth embedding,
\item the complex vector bundle $\pi^* (TM \otimes \C)$ is isomorphic to $TL \otimes \C$. 
\end{enumerate}
Moreover the first two conditions are equivalent and imply the third.
\end{proposition}
\begin{proof}
The first condition holds for an exact Lagrangian embedding $L\to T^* M$ since a Lagrangian is totally real.
It implies the second condition in view of the h-principle for Lagrangian immersions since a totally real embedding
can be made into a formal Lagrangian immersion (the totally real Grassmannian deformation retracts on the Lagrangian
Grassmannian). To see that the second condition implies the first, consider an embedding $\psi\colon L\to T^* M$ which is regularly homotopic to a Lagrangian immersion. Since $\psi$ is formally totally real (i.e. $d\psi$ is homotopic to a totally real monomorphism in the space of monomorphisms $TL \to T(T^*M)$
covering $\psi$), the map is thus isotopic to a totally real embedding in view of the h-principle for totally real embeddings.
The second condition implies the third in view of the following isomorphisms:
\[\pi^* (TM\otimes \C) \simeq\psi^*(TT^* M)\simeq TL\oplus T^* L \simeq (TL\otimes \C).\]
\end{proof}

The third condition of Proposition~\ref{prop:soft} is equivalent to the triviality of
the map $L\to BU(n)$ classifying the complex vector bundle $(\pi^*TM -TL)\otimes \C$.
Since we are in stable range this is the same as considering the stabilized map $L\to BU$.
We do not know if this condition implies the two other ones
of Proposition~\ref{prop:soft} in the general case but here are some cases
where this holds. The arguments are essentially extracted from \cite{A88}.

\begin{proposition}\label{prop:softmore}
Let $(L,\pi) \in \SS^s(M)$ such that $\pi^* (TM \otimes \C)\simeq TL \otimes \C$.
If one of the following conditions holds then $\pi$ is homotopic to a totally real embedding $L\to T^*M$:
\begin{enumerate}
\item $M$ is orientable and even-dimensional,
\item $M$ is odd-dimensional, and at least 5 dimensional and stably parallelizable.
\end{enumerate}
\end{proposition}
\begin{proof}
In view of the h-principle for Lagrangian immersions, the bundle isomorphism implies that $\pi$ is homotopic
to a Lagrangian immersion $L\to T^*M$.

Assume that $M$ (and therefore $L$) has even dimension $n$ and is orientable. The double points of an immersion of $L$ in $T^* M$
can be counted algebraically as an integer $d(L)$. Since $[L]=[M]$ (with compatible choices of orientations)
in $H_n(T^* M,\Z)$, the self-intersection of $L$ can be expressed as
\[(-1)^{n/2}\chi(M)=[M]\cdot [M]=[L]\cdot [L]=2 d(L) + e(\nu L),\]
where $e(\nu L)$ is the Euler number of the normal bundle of $L$ in $T^*M$, which due to the Lagrangian condition,
is equal to $e(\nu L)=(-1)^{n/2}\chi(L)=(-1)^{n/2}\chi(M)$ since $L$ and $M$ are homotopy equivalent (the isomorphism $\nu L \to TL$ multiplies the orientation by $(-1)^{\frac{n(n-1)}{2}}$ which is $(-1)^{\frac{n}{2}}$ when $n$ is even).
So we conclude that $d(L)=0$ and that $L\to T^*M$ is regularly homotopic to an embedding by the Whitney trick.

In the second case where $n=2k+1$, $d(L)$ is an integer modulo $2$ and so we need only prove that $L$ has an even number of double points.
From Theorem 0.5 (a) of \cite{A88}, we know that all Lagrangian immersions of $M$ and $L$ in $\C^n$ have $d_{\C^n}(M)=\chi_{\Z/2}(M)=\chi_{\Z/2}(L)=d_{\C^n}(L)$,
where $\chi_{\Z/2} =\sum_{i=0}^k \dim H_i(-;\Z/2)$ is the Kervaire semi-characteristic. Now using a tubular neighborhood of an immersion $M\to\C^n$, we obtain
from the immersion $L\to T^* M$ an immersion in $L\to \C^n$ with $d_{\C^n}(L)=d(L)+d_{\C^n}(M)$ (since $L\to M$ has odd degree).
So we obtain $d(L)=0$ and we conclude with the Whitney trick as before.
\end{proof}

\begin{remark}\label{rem: is soft?}
When $M$ is a homotopy sphere, $L$ and $M$ are stably parallelizable \cite{KM} so the bundles $\pi^*TM\otimes \C$ and $TL\otimes \C$
are stably isomorphic as complex vector bundles, and since we're in the stable range this implies that they're isomorphic as complex vector bundles before stabilization.
So Proposition~\ref{prop:softmore} says that there are no soft obstructions for an element of $\SS^s(M)$ to lie in $\SS^s_\NL(M)$.

For other manifolds $M$, there may well exist elements of $\SS^s(M)$ in a different
\emph{tangential} homotopy type which moreover have non trivial complexification,
leading to soft restrictions on $\SS^s_\NL(M)$. It would be the case if $\pi^*TM - TL$
has non-zero Pontryagin classes, for example for certain fake complex projective spaces.\end{remark}

In the remainder of this section we will use our main theorem to derive obstructions for elements of $\SS^s(M)$ to lie in $\SS^s_\NL(M)$ in various concrete examples, which we will ensure are hard by means of Propositions \ref{prop:soft} and \ref{prop:softmore}.


\subsection{Homotopy spheres}\label{sec: htpy spheres}


Suppose that $M$ is a homotopy sphere and $\pi:L \to M$ is a homotopy equivalence. As mentioned in Remark \ref{rem: is soft?} there are no soft obstructions for an element of $\SS^s(M)$ to lie in $\SS_{\text{NL}}^s(M)$ when $M$ is a homotopy sphere.

Since all homotopy spheres are PL equivalent, given any marking $S^n \to M$, we may identify $[M,PL/O]$ with $\Theta_n$ and identify the map $[M,PL/O] \to [M,G/O]$ with the Kervaire-Milnor map $\Theta_n \to \coker(J_n)$, where $J_n: \pi_nO \to \pi_n^s$ is the $J$-homomorphism. We recall that the Kervaire-Milnor map $\Theta_n \to \coker(J_n)$ can be thought of as the map which sends a homotopy sphere (which is always stably parallelizable) to the collection of framed bordism classes determined by all possible stable framings.


The kernel of the Kervaire-Milnor map is the subgroup $bP_{n+1} \subset \Theta_n$ consisting of those homotopy spheres that bound a parallelizable manifold, so the induced map on the quotient $\Theta_n/bP_{n+1} \to \coker(J_n)$ is a monomorphism (whose image is always either the whole $\coker(J_n)$ or an index two subgroup). The normal invariant of a homotopy equivalence of homotopy spheres $\Sigma_0 \to \Sigma_1$ can be thus identified with the difference between the classes of $\Sigma_0$ and $\Sigma_1$ in $\Theta_n / bP_{n+1} \subset \coker(J_n)$.

Let $\Sigma_0$ and $\Sigma_1$ be $n$-dimensional homotopy spheres. The above dicussion, together with Corollary \ref{cor:main torsion} implies:
\begin{corollary}
If there exists a Lagrangian embedding $\Sigma_1 \subset T^*\Sigma_0$, then  the class $\Sigma_0-\Sigma_1$ in $\Theta_n/bP_{n+1}$ is 2-torsion. \qed
\end{corollary}

Since $\Theta_n/bP_{n+1}$ is not in general 2-torsion,  we easily obtain many examples of hard obstructions in $\SS^s_\NL(S^n)$. For example $\Theta_{10}/bP_{11}\simeq \bZ/6$, $\Theta_{18}/bP_{19}\simeq \bZ/8 \oplus \bZ/2$ and $\Theta_{20}/bP_{21} \simeq \bZ/24$ contain lots of elements which are not 2-torsion.

Furthermore, for homotopy spheres our result from Theorems \ref{thm:main derivative} and \ref{thm:main duality} that the normal invariant $\Sigma_0 \to G/O$ of the homotopy equivalence $\Sigma_0 \to \Sigma_1$ factors through $B(G/O)$ gives additional restrictions on the possible classes in $\Theta_n/bP_{n+1}$ beyond the condition of being 2-torsion. For example, $\Theta_8/bP_9 \simeq \bZ/2$ is 2-torsion, so the knowledge that the normal is 2-torsion by itself does not yield new information. However, $\pi_7G/O=0$ and so if $\Sigma \subset T^*S^8$ is a homotopy 8-sphere it follows from our work that the normal invariant of the projection $\Sigma \to S^8$ is trivial. This implies that the class of $\Sigma$ in $\theta_8/bP_9$ is trivial. Since $bP_9=0$ this in fact implies the following consequence:

\begin{corollary}
Let $\Sigma \subset T^*S^8$ be a Lagrangian homotopy sphere. Then $\Sigma$ is diffeomorphic to $S^8$.
\end{corollary}

More generally, $bP_{n+1}=0$ whenever $n$ is even, so for even dimensional Lagrangian homotopy spheres of dimension $\neq 4$ our above results can be stated at the level of $\theta_n$. For example, we have the following statement.

\begin{corollary} $\Sigma \subset T^*S^n$ is a Lagrangian homotopy sphere and $n$ is even, then $\Sigma \# \Sigma$ is diffeomorphic to the standard $n$-sphere $S^n$.  \end{corollary}

Furthermore, as explained in Section \ref{sec:intro}, combining the present article with the results of \cite{BW88} one immediately obtains the stronger conclusion of Corollary \ref{cor:main eta2}, namely that the class of $\Sigma_0-\Sigma_1$ in $\Theta_n/bP_{n+1}$ is in the image of the composition
$$ \pi_{n-1}^s \xrightarrow{\eta} \pi_n^s \to \text{Coker}(J_n),$$ where the first map is multiplication by $\eta \in \pi_1^s$. From this a number of additional constraints can be deduced, see for example the first picture in \cite{H} for an illustration of the relevant multiplicative structure of the stable homotopy groups of spheres in small degrees.

\subsection{A general criterion}

For many applications it is sufficient to use the consequence of our results that the normal invariant $M \to G/O$ of a nearby Lagrangian $L \to T^*M$ is 2-torsion (this is Corollary~\ref{cor:main torsion}). Concretely, one may use a simple criterion, formulated in Proposition \ref{lem: criterion} below, to obtain examples of hard obstructions in $\SS^s_\NL(M)$ for a general class of smooth manifolds $M$.

Recall from \cite{HM74} (introduction to part II) that there is a bijection between the set of smoothings of a PL manifold $M$ (that has at least one smoothing) up to concordance and $[M, PL/O]$, where the bijection can be arranged to carry any fixed smooth structure on $M$ to the trivial homotopy class. For example, in the case of homotopy spheres every exotic sphere is PL equivalent to the standard sphere and hence $[S^n, PL/O]$ can be identified with the group $\Theta_n$ as was already mentioned.

\begin{proposition}\label{lem: criterion}
Suppose that $M$ is a smooth closed manifold such that $[M, BO]$ has no odd torsion. Then any odd torsion element of $[M,G/O]$ can be lifted to an element $[\pi:L \to M] \in \SS^s(M)$ such that $\pi^*TM \otimes \mathbf{C} \simeq TL \otimes \mathbf{C}$ as complex vector bundles, yet $[\pi:L \to M] \not\in \SS^s_\NL(M)$. 
\end{proposition}

\begin{proof}
  It was shown by Sullivan that away from the prime 2 the space $G/PL$ has the same homotopy type as $BO$. Furthermore, there is a map $G/PL[\frac{1}{2}] \to BO^{\otimes}[\frac{1}{2}]$ which is an equivalence of h-spaces, where $BO^{\otimes} = BO \times (1) \subset BO \times \bZ$  denotes $BO$ with the h-space structure given by tensor product of virtual vector bundles of rank 1. Our assumption was that $[M,BO]$ has no odd torsion, where the h-space structure on $BO$ is that of Whitney sum, but from the uniqueness of h-space structures on $BO$ away from the prime 2, see \cite{AP}, it follows that $[M,BO^{\otimes}[\frac{1}{2}]]$ also has no odd torsion. We conclude that $[M,G/PL]$ has no odd torsion. 
  
From the fibration $PL/O \to G/O \to G/PL$ it then follows that the associated homomorphims of abelian groups $[M,PL/O] \to [M,G/O]$ surjects onto the odd torsion subgroup of $[M,G/O]$. In particular it follows that any odd torsion class in $[M,G/O]$ automatically lies in the image of $[M, PL/O] \to [M, G/O]$.
  
  The surgery obstruction $[M, G/O] \to L^s_n(\pi_1M)$ of a class which lifts to $[M,PL/O]$ must be trivial, precisely because it comes from a smoothing of $M$. Hence it can be lifted to a class $[\pi:L \to M] \in \SS^s(M)$. To prove that $\pi^*TM \otimes \mathbf{C} \simeq TL \otimes \mathbf{C}$ as complex vector bundles it suffices to prove the stronger property that $\pi^*TM \simeq TL$ are isomorphic as real vector bundles. But we know that the map $M \to BO$ classifying $\pi^*TM - TL$ must be trivial, since it is the image of an odd torsion element under the map $[M, G/O] \to [M, BO]$ and $[M, BO]$ has no odd torsion by assumption. 
  
  Finally, the class we started with in $[M, G/O]$ is not 2-torsion, since it is an odd torsion element. Hence by Corollary \ref{cor:main torsion} we conclude that $[\pi:L \to M]$ does not lie in $\SS^s_\NL(M)$ as required.
\end{proof}

\begin{remark}\label{remark: hard}
  In view of Proposition \ref{prop:softmore}, when $L$ and $M$ are orientable of even dimension or when $L$ and $M$ are odd dimensional and parallelizable, a class $[\pi:L \to M] \in \SS^s(M)$ as in Proposition \ref{lem: criterion} can be realized by a totally real embedding, hence at least in this case the obstruction of Proposition \ref{lem: criterion} is hard.
\end{remark}

\subsection{Homotopy products of spheres}
Let us discuss the case where $M$ is homotopy equivalent to a product of two spheres $S^a \times S^b$. Recall that if $X=\Omega Y$, then $[S^a \times S^b, X] \simeq \pi_a X \oplus \pi_b X \oplus \pi_{a+b} X$. Indeed, the attaching maps for the top cell in the natural CW structure of $S^a \times S^b$ becomes trivial after suspension. Hence
\[ [S^a \times  S^b , X]  = [S^a \times  S^b , \Omega Y]  \simeq [ \Sigma \left(S^a \times S^b\right) , Y] \]
\[ \simeq [S^{a+1} \lor S^{b+1} \lor S^{a+b+1}, Y]  \simeq \pi_{a+1} Y \oplus \pi_{b+1}Y \oplus \pi_{a+b+1}Y \]
\[ \simeq \pi_a X \oplus \pi_b X \oplus \pi_{a+b} X . \]

Therefore, for $i=a, b$ and $a+b$, we deduce that the group $\Theta_i / bP_{i+1}$ is a summand of the image of $[S^a \times S^b, PL/O] \to [S^a \times S^b, G/O]$. Indeed, since $PL/O$ and $G/O$ are loop spaces,  we have
\[ [S^a \times S^b , PL/O] \simeq  \pi_a PL/O \oplus \pi_b PL/O \oplus \pi_{a+b} PL/O \] 
\[ [S^a \times S^b , G/O] \simeq  \pi_a G/O \oplus \pi_b G/O \oplus \pi_{a+b} G/O \]  
The map $[S^a \times S^b , PL/O] \to [S^a \times S^b , G/O]$ respects this decomposition and is just the Kervaire-Milnor map $\Theta_i \to \text{coker}(J_i)$ on each factor, from which the claim follows.

Since $BO$ is a (infinite) loop space and $[S^i , BO] = \pi_i BO$ has no odd torsion, it follows by the same reasoning that $[S^a \times S^b, BO]$ has no odd torsion. As we saw before, for $n=10,18,20,...$ the group $\Theta_n / bP_{n+1}$ is not 2-torsion. Hence by Proposition \ref{lem: criterion} we obtain hard obstructions on $\SS^s_\NL(S^a \times S^b)$ when $a$, $b$ or $a+b$ are equal to $10,18,20,...$.

Suppose  next that $M$ is homotopy equivalent to $T^n$, the product of $n$ copies of $S^1$. As before, if $X=\Omega Y$ be a loop space, then $[T^n, X] \simeq \bigoplus_{i=0}^n \pi_iX^{\oplus {n \choose i}}$. Hence once again $[T^n, BO]$ has no odd torsion.

Just as before, for $i \leq n$, the group $\Theta_i / bP_{i+1}$ is a summand of the image of $[T^n, PL/O] \to [T^n, G/O]$. Since $PL/O$ and $G/O$ are loop spaces, reasoning as before we conclude
\[ [T^n , PL/O] \simeq  \bigoplus_i (\pi_iPL/O)^{\oplus {n \choose i}} , \qquad [T^n , G/O] \simeq  \bigoplus_i (\pi_i G/O)^{\oplus {n \choose i}} \]  
Moreover, the map $[T^n , PL/O] \to [T^n , G/O]$ respects this decomposition and is just the Kervaire-Milnor map $\Theta_i \to \text{coker}(J_i)$ on each factor.

Hence by Proposition \ref{lem: criterion}  for all $n \geq 10$ there are hard obstructions on $\SS^s_\NL(T^n)$. Clearly similar considerations apply to arbitrary products of spheres.

One could push this a bit further using the fact that the twisted duality map $B(G/O) \to G/O$ may be realized as a loop map (indeed it may be realized as an infinite loop map). We only construct the twisted duality map as a map of spaces in the present article. However, assuming this we may obtain hard obstructions beyond those provided by Proposition \ref{lem: criterion}, for example if $M$ is any product $S^8 \times S^b$ or any torus $T^n$ for $n \geq 8$, then the image of $[M,PL/O] \to [M,G/O]$ contains a $\Theta_8/bP_9 \simeq \Theta_8 \simeq \bZ/2$ summand which is not in the image of the map $[M,B(G/O)] \to [M,G/O]$ because the twisted duality map $B(G/O) \to G/O$ respects the above decomposition (by virtue of being a loop map) and $\pi_7G/O=0$. This class lifts to an element of $\SS^s(M)$ which by Theorems \ref{thm:main derivative} and \ref{thm:main duality} is not in $\SS_{\text{NL}}^s(M)$.

\subsection{Fake complex projective spaces}

As a final application we briefly discuss the case in which $M$ is homotopy equivalent to complex projective space.

First, we recall that $[\bC \mathbf{P}^n, BO]$ has no odd torsion, see for instance \cite{F67}. Therefore, in view of Remark \ref{remark: hard}, in order to apply Proposition \ref{lem: criterion} and exhibit hard obstructions on $S^s_{\text{NL} }(\bC \mathbf{P}^n)$  it suffices to exhibit odd torsion elements in $[\bC \mathbf{P}^n, G/O]$. 

Concrete examples of such odd torsion may be obtained by considering the Atiyah-Hirzebruch spectral sequence for the stable cohomotopy of $\bC \mathbf{P}^n$, using knowledge of the stable homotopy groups of spheres together with relevant information about attaching maps and homotopy group compositions. Indeed, as explained by Brumfiel  in  \cite{B68},  the torsion subgroup of $[\bC \mathbf{P}^n , G/O]$ may be identified with the zeroth stable cohomotopy group of $\bC \mathbf{P}^n$, from which it follows that  $[\bC \mathbf{P}^n, G/O]$ contains a $\bZ/3$ summand for $n=5$ and $n=6$.  See \cite{B68} for further results in this direction indicating that odd torsion in $[\bC \mathbf{P}^n, G/O]$ is plentiful.


\appendix
\section{Two-sided bar construction}

Let $G$ be a (well-pointed) group like topological monoid, $X$ a right $G$-space and $Y$ a left $G$-space.
We denote $B_\bullet(X,G,Y)$ the associated two-sided bar construction. This is a simplicial space and we denote $B(X,G,Y)$
its geometric realization. We will use the following well-known results (see \cite{May}).

\begin{proposition}\label{lem:bar0}
The natural map \[B(X,G,G)\to X\] is a homotopy equivalence.
\end{proposition}

\begin{proposition}\label{lem:bar1}
The map 
\[B(X,G,Y)\to B(X,G)\]
is a quasi-fibration with fiber weakly homotopy equivalent to $Y$.
\end{proposition}

\begin{corollary}\label{lem:bar2}
Let $E$ be a right $G$-space and let $B$ be a space with trivial $G$ action. Assume that $f:E \to B$ is a $G$-invariant surjective map which is a quasi-fibration. Suppose that for each $e\in E$, the map $G\to f^{-1}(f(e))$ given by $g\mapsto e\cdot g$ is a weak homotopy equivalence. Then the induced map $B(E,G)\to B$ is a weak homotopy equivalence.
\end{corollary}

\begin{proof}
  Consider the following commutative diagram
  \begin{center}
    \begin{tikzcd}
      {B(E,G,G)} \arrow[r,"\sim"]\arrow[d]& E  \arrow[d]\\
      {B(E,G)}\arrow[r] & B
    \end{tikzcd}
  \end{center}
  The top horizontal map is a homotopy equivalence according to Lemma~\ref{lem:bar0}.
  The left vertical map is a quasi-fibration with fiber weakly homotopy equivalent to $G$ according to Lemma~\ref{lem:bar1}
  and so is the right vertical map by assumption. Moreover by $G$-equivariance the horizontal maps
  induce a weak homotopy equivalence on the fibers of the vertical maps. Hence the map $B(E,G)\to B$
  is also a weak homotopy equivalence.
\end{proof}

Next, we let $A$ be a topological monoid (not necessarily grouplike) and $X$ a space with a right action of $A$.
\begin{proposition}[c.f. Proposition D.1 of \cite{H11}] \label{lem:bar3}
If $A$ acts on $X$ by weak homotopy equivalences, namely for all $m\in A$, $x\mapsto x\cdot m$ is a weak homotopy equivalence of $X$,
then $B(X,A)\to BA$ is a quasifibration with fiber $X$. \qed
\end{proposition}
Up to homotopy, the above implies that the $A$-action on $X$ extends to an action of its group completion $\Omega B A$, which can be used to apply results about actions of group-like monoids to this setting. We can avoid the up-to-homotopy replacement in proving the following:
\begin{corollary}\label{lem:bar4}
Let $\phi:A\to A'$ be a monoid map such that $BA\to BA'$ is a homotopy equivalence.
Let $X$ and $X'$ be spaces with right actions by weak homotopy equivalences of $A$ and $A'$ respectively.
Let $f:X\to X'$ be a weak homotopy equivalence which is $\phi$-equivariant. Then the induced map $B(X,A)\to B(X',A')$
is a weak homotopy equivalence.
\end{corollary}
\begin{proof}
Consider the following diagram where the vertical sequences are fibration sequences according to Corollary~\ref{lem:bar3}
\begin{center}
\begin{tikzcd}
X \arrow[r]\arrow[d] & X'\arrow[d]\\
B(X,A)\arrow[r] \arrow[d]& B(X',A')\arrow[d]\\
BA \arrow[r] & BA'
\end{tikzcd}
\end{center}
The result then follows from the five lemma.
\end{proof}

\bibliographystyle{hplain}

\end{document}